\documentclass[11pt,leqno]{amsart}
\usepackage{amssymb}
\usepackage{amsfonts}
\usepackage{amsthm}
\usepackage{mathtools}
\usepackage{todonotes}

\usepackage{amsthm}

\usepackage[pagebackref,colorlinks,citecolor=blue,linkcolor=blue]{hyperref}

\usepackage{cite}   

\usepackage{amsmath,bbm,amssymb,amsxtra}

\makeatother
\makeatletter
\let\Enumerate=\enumerate
\renewcommand{\enumerate}{\Enumerate%
	\setlength{\@topsep}{0pt}
	\setlength{\itemsep}{0pt}%
	\setlength{\parskip}{0pt plus 1pt}%
	\renewcommand{\theenumi}{\textup{(\alph{enumi})}}%
	\renewcommand{\labelenumi}{\theenumi}%
}
\let\endEnumerate=\endenumerate
\renewcommand{\endenumerate}{\endEnumerate\unskip}

\def\Xint#1{\mathchoice
	{\XXint\displaystyle\textstyle{#1}}
	{\XXint\textstyle\scriptstyle{#1}}
	{\XXint\scriptstyle\scriptscriptstyle{#1}}
	{\XXint\scriptscriptstyle\scriptscriptstyle{#1}}
	\!\int}

\def\XXint#1#2#3{{\setbox0=\hbox{$#1{#2#3}{\int}$ }
		\vcenter{\hbox{$#2#3$ }}\kern-.6\wd0}}
\def\dashint{\Xint-}

\newtheorem{theorem}{Theorem}[section]
\newtheorem{lemma}[theorem]{Lemma}
\newtheorem{proposition}[theorem]{Proposition}

\newtheorem{corollary}[theorem]{Corollary}

\numberwithin{equation}{section}

\theoremstyle{definition}
\newtheorem{example}[theorem]{Example}

\DeclareMathOperator{\capp}{cap}
\newcommand{\caps}{\capp_{s,1}}
\newcommand{\loc}{_{\rm loc}}

\DeclareMathOperator*{\esssup}{ess\,sup}

\DeclareMathOperator{\diam}{diam}
\DeclareMathOperator{\supp}{spt}

\newcommand{\BV}{\mathrm{BV}}

\title[Weak Harnack inequality and Cartan property]{Weak Harnack inequality and Cartan property for   nonlocal $W^{s,1}$-minimizers}

\author{Panu Lahti} 
\address{Academy of Mathematics and Systems Science, Chinese Academy of Sciences, Beijing 100190, PR China}\email{panulahti@amss.ac.cn}
\author{Yuxin Li} 
\address{Academy of Mathematics and Systems Science, Chinese Academy of Sciences, Beijing 100190, PR China}\email{ liyuxin@amss.ac.cn}
\author{Khanh Nguyen} 
\address{Academy of Mathematics and Systems Science, Chinese Academy of Sciences, Beijing 100190, PR China}\email{  khanhnguyen@amss.ac.cn}


\begin{document}
	
	\maketitle
	\begin{center}
		March 22, 2026
	\end{center}

	\begin{abstract}
		We establish a 
		weak Harnack inequality for  nonlocal
		$W^{s,1}$-subminimizers 
		in a complete, connected, doubling metric
		measure space where $0<s<1$. As a
		corollary, we prove that $W^{s,1}$-subminimizers are semicontinuous,
		up to a suitable choice of pointwise representative.
		We then prove \emph{Cartan-type properties}
		for $W^{s,1}$-superminimizers.
		The theory turns out to be mostly analogous with the local case
		of $\BV$ super- and subminimizers.
		Our results seem to be new even in the classical Euclidean setting.
	\end{abstract}
	
	\tableofcontents
	
	\section{Introduction}
	
	In 2016, Di Castro--Kuusi--Palatucci \cite{DKP16} studied a weak Harnack inequality for the following problem:
	for an open set $\Omega$ in $\mathbb R^n$ and
	$f\in W^{s,p}(\mathbb R^n)$, where $n\in\mathbb N$, $0<s<1$, and $1<p<\infty$,
	\begin{equation}\label{eq1.1-14oct}
		\begin{cases}
			(-\Delta)^s_p u\, \leq\,  0&\ \ {\rm in\ \ }\Omega,\\
			u=f &\ \ {\rm in\ \ } \mathbb R^n\setminus\Omega,
		\end{cases}
	\end{equation}
	where the symbol $(-\Delta)^s_p$ denotes the standard fractional $p$-Laplacian operator.
	Here $W^{s,p}(A),$ where
	$A\subset\mathbb R^n$ is measurable, denotes  the collection of all $p$-integrable functions
	$f$ on $A$ for which the following semi-norm
	\[
	\|u\|_{\dot W^{s,p}(A)}:= \left(\int_{A}\int_{A}\frac{|u(x)-u(y)|^p}{|x-y|^{n+ps}}\,dx\,dy\right)^{\frac{1}{p}}
	\]
	is finite. Denote by $W^{s,p}_0(A)$ the space consisting of all functions $u$ in $W^{s,p}(\mathbb R^n)$ for which $u=0$
	almost everywhere (a.e.) on $\mathbb R^n\setminus A$. We say that $u$ is a weak subsolution to
	the problem \eqref{eq1.1-14oct} if $u=f$ a.e. on $\mathbb R^n\setminus\Omega$ and for every
	nonnegative function $\eta\in W^{s,p}_0(\Omega)$,
	\[
	\int_{\mathbb R^n}\int_{\mathbb R^n} \frac{|u(x)-u(y)|^{p-2}(u(x)-u(y))(\eta(x)-\eta(y))}{|x-y|^{n+ps}}\,dx\,dy\geq 0.
	\]
	By \cite[Theorem 1.1]{DKP16}, for every $0<\delta<1$ and
	every weak subsolution to the problem \eqref{eq1.1-14oct}, there is a constant $C>0$ depending only on $n,p,s$ such that for all balls $B(x_0,r)\subset\Omega$ with radius $r>0$ and center at $x_0$, 
	\begin{align*}
		&\quad \esssup_{y\in B(x_0,r/2)} u(y) \\
		&\leq \delta \cdot {\rm Tail}{(u_+,x_0,r/2)} + C\cdot \delta^{-\frac{(p-1)n}{sp^2}} \left( \frac{1}{|B(x_0,r)|}\int_{B(x_0,r)} u_+^p(x)\,dx \right)^{\frac{1}{p}},
	\end{align*}
	where $u_+:=\max\{u,0\}$ and 
	\[
	{\rm Tail}(u,x,r):= \left( r^{ps}\int_{\mathbb R^n\setminus B(x,r)} \frac{|u(y)|^{p-1}}{|x-y|^{n+ps}}\,dy \right)^{\frac{1}{p-1}}.
	\]
	
	We have that $u$ is a weak subsolution to  the problem \eqref{eq1.1-14oct} if and only if $u$ is 
	subminimizer of the following functional
	\begin{equation}\label{eq1.2-7dec}\notag
		F_{s,p}(u):= \left( \int_\Omega\int_\Omega +2\int_\Omega\int_{\Omega^c} \right) \frac{|u(x)-u(y)|^{p}}{|x-y|^{n+ps}}\,dx\,dy,
	\end{equation}
	i.e. $F_{s,p}(u)\leq F_{s,p}(u+\eta)$ for every nonpositive function $\eta\in W^{s,p}_0(\Omega)$,
	see \cite[Theorem 2.3]{DKP16}. Here $\Omega^c:=\mathbb R^n\setminus \Omega$.
	The aim of this paper is to study a weak Harnack inequality for subminimizers, or alternatively
	weak subsolutions to the problem \eqref{eq1.1-14oct}, in the case of $p=1$. 

	It is natural to consider the problem in the more general setting of metric measure spaces $(X,d,\mu)$.
	Let $\Omega\subset X$ be an open set.
	For $0<s<1$, we consider the following functional $F$ of two measurable functions $u,v$:
	\begin{align*}
		F(u,v)
		:= \left( \int_\Omega\int_\Omega +2 \int_\Omega\int_{\Omega^c} \right) \frac{|u(x)-u(y)| 
			-|v(x)-v(y)| }{\mu(B_{x,y})d(x,y)^s}\,d\mu(x)\,d\mu(y),
	\end{align*}
	where $B_{x,y}:=B(x,d(x,y))\cup B(y,d(y,x))$ and $\Omega^c:=X\setminus\Omega$.
	See Subsection \ref{sec2.1} for more discussion on this functional.  A measurable
	function $u:X\to[-\infty,\infty]$ is said to be a \emph{nonlocal $W^{s,1}$-sub(super)minimizer}
	{in $\Omega$,}
	or shortly an \emph{$s$-sub(super)minimizer} 
	{in $\Omega$,}
	if
	$u\in \dot W^{s,1}(\Omega)$ and for every nonpositive (nonnegative) function
	$\eta\in W^{s,1}_0(\Omega)$, we have
	\begin{equation}\notag
		F(u,u+\eta)\leq 0.
	\end{equation}
	{It is an \emph{$s$-minimizer} if the above holds for every $\eta\in W^{s,1}_0(\Omega)$.}
	Here $W^{s,1}_0(\Omega)$ and $\dot W^{s,1}(\Omega)$ are defined as in
	Subsection \ref{sec2.1}.
	
	We denote by $B(x,r)$ the open ball with radius $r>0$ and center at $x\in X$.
	We say that $(X,d,\mu)$ is a metric measure space equipped with a \emph{doubling measure}  if there exist a doubling constant $C_\mu>0$ such that
	$
	0<\mu(B(x,r))<\infty 
	$
	and
	$
	\mu(B(x,2r))\leq C_\mu \mu(B(x,r)) 
	$
	for all balls $B(x,r)$. A positive real constant $Q>0$ is called a \emph{doubling dimension} if there is a constant $c>0$ such that
	\begin{equation}\notag
		{\frac{\mu(B(x,r))}{\mu(B(x,R))}\geq 
			c
			\left( \frac{r}{R}\right)^Q}
	\end{equation}
	for all $0<r<R$ and $x\in X$.

	Our first main result is a weak Harnack inequality for  subminimizers 
	{in a metric space.}
	This extends the weak Harnack inequality
	\cite[Theorem 1.1]{DKP16} to the case $p=1$.
	
	\begin{theorem}[Weak Harnack inequality]
		\label{thm1}
		Let $(X,d,\mu)$ be a complete, connected metric measure space equipped with a doubling measure.
		Let $0<s<1$ and let $Q$ be a doubling dimension such that $Q>s$.
		Let $\Omega\subset X$ be a nonempty open set and let $u$ be an $s$-subminimizer 
		in $\Omega$.
		
		Then there exists a constant $C>0$ depending only on $s$ and the doubling constant $C_\mu$ such that for every $k_0\in\mathbb R$,
		every $x_0\in\Omega$, and all $0<r<R$ with $B(x_0,R)\subset\Omega$,
		\[
		\esssup_{y\in B(x_0,r)}(u(y)-k_0)
		\le C\left(\frac{R}{R-r}\right)^Q
		\dashint_{B(x_0,R)}(u-k_0)_+\,d\mu.
		\]
		
	\end{theorem}
	
	The weak Harnack inequality that we establish in Theorem~\ref{thm1} does not involve
	any tail term. If we additionally assume that $\mu$ supports a $1$-Poincar\'e inequality, see the definition in Subsection \ref{sec2.2},
	then the constant $C$ in Theorem \ref{thm1} can be taken to be $(C'/s)^{Q/s}$, which is uniformly
		bounded as $s\to 1^-$;
	here $C'>0$ depends only on the doubling and Poincar\'e constants, see Theorem \ref{thm1-20nov}.
	This stability of the constant is essentially obtained
	from the sharp fractional Poincar\'e inequality proved recently in \cite{KLLZ25}.
	
	As a consequence of 
	the weak Harnack inequality, with some further work,
	one usually obtains the continuity of solutions. By  \cite[Theorem 1.2]{DKP16}, 
	any solution to the problem \eqref{eq1.1-14oct} is locally H\"older continuous on $\Omega$.  This result is a counterpart to the H\"older regularity theory of 
	$p$-harmonic functions $u$ when $1<p<\infty$, i.e. $-\Delta_pu=0$, see for instance \cite{BB11,HKM06}. Such continuity  is too much to hope for in the fractional case because 
	there is an $s$-minimizer ${\bf 1}_{S}$ , $S\subset \mathbb{R}^2$,
	in an open set
	on the plane for which
	the boundary
	$\partial S$ is a nonlocal minimal surface, see \cite{CRS10,Kl25} for a discussion
	on nonlocal minimal surfaces. 
	Nevertheless, using the weak Harnack inequality of Theorem \ref{thm1},
	we can obtain \emph{semicontinuity} of 
	{$s$-sub(super)minimizers,}
	as in the corollary below. 
	
	Let $u:X\to [-\infty,\infty]$.
	The \emph{lower and upper approximate limits} of $u$ at $x\in X$
	are defined respectively by 
	\begin{equation}\label{eq3.19-2011}
		u^\wedge(x):= \sup\left\{ t\in\mathbb R: \lim_{r\to 0^+}\frac{\mu(\{y\in B(x,r): u(y)<t\}) }{\mu(B(x,r))}=0 \right\}
	\end{equation}
	and
	\begin{equation}\label{eq3.20-2011}
		u^\vee(x):= \inf\left\{ t\in\mathbb R: \lim_{r\to 0^+}\frac{\mu( \{y\in B(x,r): u(y)>t\}) }{\mu(B(x,r))}=0 \right\}.
	\end{equation}
	
	\begin{corollary}\label{cor3.3-16oct}
		Let $(X,d,\mu)$ be a complete, connected metric measure space equipped with a doubling measure. Let $0<s<1$. Let $\Omega\subset X$ be a nonempty open set. 
		Then $u^\wedge$ is lower semicontinuous on $\Omega$ for every $s$-superminimizer $u$ in $\Omega$, and $u^\vee$ is upper semicontinuous on $\Omega$ for every $s$-subminimizer $u$ in $\Omega$.
	\end{corollary}
	
	A set $A\subset X$
	is said to be \emph{$s$-thin} at $x\in X$ if 
	\[
	\lim_{r\to0^+} \frac{\caps(A\cap B(x,r), B(x,2r))}{\caps(B(x,r),B(x,2r))} =0.
	\] 
	Here the condenser $(s,1)$-capacity ${\rm cap}_{s,1}$ is defined as in Subsection \ref{subsection5.1}.
	
	With the help of the weak Harnack inequality of Theorem \ref{thm1},
	we are able to prove the following \emph{Cartan-type properties} for superminimizers.
	
	\begin{theorem}[Weak Cartan property]\label{weakCartan-main}
		Let $(X,d,\mu)$ be a complete, connected metric measure space equipped with a doubling measure.  
		Let $0<s<1$.
		Let $A\subset X$ be $s$-thin at $x\in \overline A \setminus A$. 
		Then there exist
		$k\in\mathbb N$ depending only on $s$ and the doubling constant, $R>0$, and
		sets $ E_0, E_1, \ldots, E_{k-1}\subset X$ that satisfy the following properties:
		\begin{enumerate}
			\item \label{a-item-} each ${\bf 1}_{E_{j}}$ is an $s$-superminimizer in $B(x,R)$ for each $j=0,\ldots,k-1$; 
			
			\item \label{b-item-} ${\bf 1}_{E_{j}}\in W^{s,1}(X)$ for each $j=0,\ldots,k-1$;
			
			\item \label{c-item-} $\max\{{\bf 1}_{E_0}^\wedge, {\bf 1}_{E_2}^{\wedge},\,\ldots, {\bf 1}_{E_{k-1}}^\wedge\}\equiv 1$ on $A\cap B(x,R)$;
			
			\item \label{d-item-} ${\bf 1}_{E_0}^\vee(x)={\bf 1}_{E_2}^\vee(x)=\ldots={\bf 1}_{E_{k-1}}^\vee(x)=0$;
			
			\item \label{e-item-} the set 
			\[
			\left\{y\in X: \max_{0\leq j\leq k-1} {\bf 1}_{E_j}^\vee(y)>0\right\}
			\]
			is $s$-thin at $x$.
		\end{enumerate}
	\end{theorem}
	
	The claim \ref{c-item-} in Theorem \ref{weakCartan-main} gives that there is at least one $j\in\{0,\ldots,k-1\}$ such that $x\in \overline {E_j}\setminus E_j$ and $\lim_{E_j\cap A\ni y\to x}{\bf 1}_{E_j}^\wedge(y)=1$,
	{in particular ${\bf 1}_{E_j}^\wedge$ fails to be upper semicontinuous at $x$}.
	In Example \ref{ex3.9-15dec}, we construct a set $A$
	in a Euclidean space that is $s$-thin at the origin $O$,
	such that  $O\in\overline{A}\setminus A$ and moreover $C_{s,1}(\{O\})>0$.
	This gives an example of a situation where Theorems \ref{weakCartan-main}
	and \ref{thm3.17-26nov} are applicable.
	
	Notice that Theorem \ref{weakCartan-main} is called a \emph{weak} Cartan property
	because we need $k$ different $s$-superminimizers ${\bf 1}_{E_j}$ whose
	\emph{union} covers $A$ in a neighborhood of $x$.
	If only one superminimizer is needed, then we use the term \emph{(strong) Cartan property}.
	If $x$ has strictly positive capacity, we can show such a property.
	
	\begin{theorem}
		[Strong Cartan property]
		\label{thm3.17-26nov}
		Let $(X,d,\mu)$ be a complete, connected metric measure space equipped with a doubling measure.  Let $0<s<1$. 
		Let $A\subset X$ be  
		{$s$-thin}
		at $x\in \overline A\setminus A$ with $C_{s,1}(\{x\})>0$. 
		Then for all $0<R_0<\frac{1}{4}\diam X$, there exists
		an $s$-superminimizer $u$ 
		{in $B(x,R_0)$}
		such that 
		\begin{equation}\notag
			\lim_{A\ni y\to x}u^\wedge (y)=\infty>u^\vee(x).
		\end{equation}
	\end{theorem}

	The Cartan properties that we prove
	are analogs of results in the local case, from which the terminology is also borrowed.
	We refer the reader to \cite{BBL15,BBL18,La20,MZ} and references therein for
	a discussion on the analogous results
	in the local case of $W^{1,p}$-superminimizers and  BV-superminimizers.

	\section{Preliminaries}
	Throughout the paper,   we always assume that $0<s<1$  and 
	that $(X,d,\mu)$ is a complete, connected metric measure space equipped with a doubling measure defined as below.
	\subsection{Nonlocal minimizers}\ \label{sec2.1}
	
	A triple $(X,d,\mu)$ is said to be a metric measure space if $(X,d)$
	is a nonempty separable metric space and $\mu$ is a locally finite Borel-regular
	outer measure on $X$.
	Hence for every
	set $A\subset X$, we have
	\[
	\mu(A)= \inf \{ \mu(O): A\subset O\subset X,\, O\  {\rm is\ open} \},
	\]
	see e.g. \cite[Proposition 3.3.37]{HKST15}.
	Then $\mu(A)$ is well-defined for every $A\subset X$.
	
	Recall that we denote by $B(x,r)$ the open ball with radius $r>0$ and center $x\in X$.
	For $\tau>0$, we set $\tau B(x,r):=B(x,\tau r)$.
	We always assume $\mu(B)>0$ for every ball $B=B(x,r)$.
	On a few occasions, we use closed balls, denoted by 
	$\overline{B}{(x,r)}:=\{y\in X:d(y,x)\le r\}$.

	For an open set $\Omega\subset X$ and
	a measurable function $u:\Omega\to [-\infty, \infty]$, we set
	\[
	\|u\|_{\dot W^{s,1}(\Omega)}=\int_{\Omega}\int_{\Omega}\frac{|u(x)-u(y)|}{\mu(B_{x,y})d(x,y)^s}\,d\mu(x)\,d\mu(y),
	\]
	where $B_{x,y}:=B(x,d(x,y))\cup B(y,d(y,x))$, and
	\[
	\|u\|_{W^{s,1}(\Omega)}:=\|u\|_{L^1(\Omega)}+\|u\|_{\dot W^{s,1}(\Omega)}.
	\]
	We denote by $ \dot W^{s,1}(\Omega)$ and $W^{s,1}(\Omega)$ the spaces consisting of all measurable functions $u$ on $\Omega$ for which  the quantities $\|u\|_{\dot W^{s,1}(\Omega)}$ and $\|u\|_{W^{s,1}(\Omega)}$ are respectively finite. We denote by $W^{s,1}_0(\Omega)$ the collection of all functions $u\in W^{s,1}(X)$ such that $u= 0$ a.e. on $\Omega^c$, where $\Omega^c:=X\setminus\Omega$.
	We also notice
	that if $\|u\|_{\dot W^{s,1}(\Omega)}<\infty$, then $u$ is locally integrable on $\Omega$;
	indeed,
	we have for a.e. $y\in \Omega$ that $u(y)\in\mathbb R$ and $\int_{\Omega}\frac{|u(x)-u(y)|}{\mu(B_{x,y})d(x,y)^s}\,d\mu(x)<\infty$. Hence for each ball $B\subset\Omega$ with radius $r$, we have that for a.e. $y\in B$,  
	\begin{align}\label{eq2.1-11dec}
		\dashint_B |u(x)|\,d\mu(x)
		&\leq  |u(y)|+\dashint_{B} {|u(x)-u(y)|}\,d\mu(x)\\
		&\leq |u(y)|+\, \frac{\mu(4B) }{\mu(B)}(2r)^s\int_{B} \frac{|u(x)-u(y)|}{\mu(B_{x,y})d(x,y)^s}\,d\mu(x)<\infty,  \notag
	\end{align}
	because
	$B_{x,y}\subset 4B$
	and $d(x,y)\leq 2r$ for all $x\in B$.
	
	Let $\Omega\subset X$ be an open set.
	We consider the following functional $F$ of
	two measurable functions $u,v$ on $X$:
	\[
	F(u,v):=\left(\int_{\Omega}\int_{\Omega}+2\int_{\Omega}\int_{\Omega^c}\right) \frac{|u(x)-u(y)|-|v(x)-v(y)|}{\mu(B_{x,y})d(x,y)^s}\,d\mu(x)\,d\mu(y).
	\]

	A measurable function $u:X\to[-\infty,\infty]$ is an \emph{$s$-superminimizer}
	(resp. \emph{$s$-subminimizer}) {in $\Omega$}
	if $u\in \dot W^{s,1}(\Omega)$ and for every nonnegative (resp. nonpositive)
	function $\eta\in W^{s,1}_0(\Omega)$, we have \begin{equation}\label{eq1.2-27oct}
		F(u,u+\eta)\leq 0.
	\end{equation}

	Note that such $u$ is assumed to belong to $\dot W^{s,1}(\Omega)$. Recall that 
	\[
	F_{s}(u):=F_{s,1}(u)= \left( \int_\Omega\int_\Omega +2\int_\Omega\int_{\Omega^c} \right) \frac{|u(x)-u(y)|}{\mu(B_{x,y})d(x,y)^s}\,d\mu(x)\,d\mu(y).
	\]
	If $F_{s}(u)<\infty$, then we have from the triangle inequality that $F_s(u+\eta)<\infty$ since $\eta\in W^{s,1}(X)$. It follows from adding $F_{s}(u+\eta)$ to both sides of \eqref{eq1.2-27oct} that if $F_s(u)<\infty$ then the condition \eqref{eq1.2-27oct} is equivalent to  the condition of the minimization problem of the functional $F_{s}$.
	Since $u\in \dot W^{s,1}(\Omega)$, we have that the condition  $F_s(u)<\infty$ is equivalent to the finiteness of the following nonlocal quantity
	\[
	\int_\Omega\int_{\Omega^c} \frac{|u(x)-u(y)|}{\mu(B_{x,y})d(x,y)^s}\,d\mu(x)\,d\mu(y).
	\]

	\subsection{Doubling measure and  Poincar\'e inequality}\
	\label{sec2.2}

	We say that $(X,d,\mu)$ is a \emph{metric measure space equipped with a doubling measure} if there exists a doubling constant $C_\mu>0$  such that
	\begin{equation}\notag
		\mu(B(x,2r))\leq C_\mu \mu(B(x,r)) 
	\end{equation}
	for all balls $B(x,r)$. 
	
	We say that  $\mu$ supports a \emph{$1$-Poincar\'e inequality} if there exist  Poincar\'e constants $C_P>0$,  $\lambda\ge1$ such that
	\begin{equation}\label{equ1.3-2011}\notag
		\dashint_{B(x,r)} |u - u_{B(x,r)}|\,d\mu
		\leq C_P r \dashint_{B(x,\lambda r)} g\, d\mu
	\end{equation}
	for every ball $B(x,r)$, every function $u$ integrable on balls, and every upper gradient $g$ of $u$, where $u_B:=\dashint_{B}u\, d\mu:=\frac{1}{\mu(B)}\int u\, d\mu.$
	Here we say that a Borel function $g:X\to[0,\infty]$ is an upper gradient of $u$ if 
	\[
	|u(\gamma(0))-u(\gamma(1))|\leq \int_\gamma g\, ds
	\]
	for every rectifiable curve $\gamma$ in $X$.
	We refer the reader to \cite{BB11,HKST15,Ha03,HK00,Sh00} for a discussion on upper gradients, Poincar\'e inequalities, and their applications in complete, doubling metric measure spaces.

	We write $a \lesssim b$ or $b \gtrsim a$ 
	if there is an implicit comparison constant $C>0$ depending only on  the doubling constant
	(and  
	{Poincar\'e constants}	
	if we assume a Poincar\'e inequality) such that $a \leq C\cdot b$,  and $a \simeq b$ if
	$a \lesssim b \lesssim a$. If $C$ depends on a distinguished and important parameter $t$,
	then we write $a\lesssim_t b$, or $b\gtrsim_t a$, or $a\simeq_t b$.
	
	A positive real constant $Q>0$ is called a \emph{doubling dimension} if there is a constant $c>0$ such that
	\begin{equation}\label{eq1.2-2011}
		{ \frac{\mu(B(x,r))}{\mu(B(x,R))}\geq {c} \left( \frac{r}{R}\right)^Q }
	\end{equation}
	for all $0<r<R$ and $x\in X$. If $\mu$ is a doubling measure with a doubling constant $C_\mu>0$, then $Q=\log_2(C_\mu)$
	satisfies \eqref{eq1.2-2011}
	with the constant $c>0$ depending only on the doubling constant $C_\mu$;
	a smaller $Q$ may work as well.

	\begin{theorem}[{\cite[Lemma 2.2]{DLV23}}]
		There exists a constant $C\ge 1$ depending only on the doubling constant so that for every ball
		$B$ with radius $r>0$ and for every measurable function $u$, 
		\begin{align}
			\dashint_{B}|u-u_{B}| \,d\mu
			\leq\,  \frac{C r^s}{\mu(B)} \|u\|_{\dot W^{s,1}(B)}.
			\label{Poincare-local-1}
		\end{align}
	\end{theorem}
	
	\begin{theorem}[{\cite[Theorem 3.4]{DLV23}}]

		Let $Q>s$ be a doubling dimension.
		
		Then there exists a constant $C\ge 1$ depending only on $s$ and the doubling constant so that for all balls $B$ with radius $r>0$ and for all measurable functions $u$, 
		\begin{align}
			\left(\dashint_{B}|u-u_{B}|^{\frac{Q}{Q-s}} \,d\mu\right)^{\frac{Q-s}{Q}} 
			\leq\, 
			\frac{C r^s}{\mu(B)} \|u\|_{\dot W^{s,1}(2 B)}.
			\label{Poincare-local-2}
		\end{align}
	\end{theorem}
	In fact, \eqref{Poincare-local-2} holds for doubling measures satisfying the reverse doubling condition of \cite[Theorem 3.4]{DLV23}, which is ensured by connectedness; see \cite[Corollary 3.8]{BB11}. Moreover, the proof of \cite[Theorem 3.4]{DLV23} shows that the constant in \eqref{Poincare-local-2} may degenerate only as $s \to 0$.
	
	Under the additional assumption that $X$ supports a $1$-Poincar\'e inequality, one obtains the following quantitatively sharp fractional Poincar\'e inequality  from \cite{KLLZ25}. We remark that the assumption of Poincar\'e inequality is stronger than connectedness since every metric measure space supporting  a Poincar\'e inequality is connected, see e.g. \cite[Section 4.5]{BB11} or \cite[Section 8]{HKST15}.
	
	\begin{theorem}[{\cite[Theorem 1.1]{KLLZ25}}]
		Let $Q>s$ be a doubling dimension. Suppose that $\mu$ supports a $1$-Poincar\'e inequality.
		
		Then there exist  constants $C\ge 1$ and $\tau\ge 1$ depending only on the doubling and Poincar\'e constants so that for all balls $B$ with radius $r>0$ and for all measurable functions $u$, 
		\begin{align}
			\, \left(\dashint_{B}|u-u_{B}|^{\frac{Q}{Q-s}} \,d\mu\right)^{\frac{Q-s}{Q}} 
			\leq\,  \frac{C(1-s)r^s}{\mu(\tau B)} \|u\|_{\dot W^{s,1}(\tau B)}.
			\label{Poincare-local}
		\end{align}
	\end{theorem}

	Moreover, as a consequence of \eqref{Poincare-local-2}-\eqref{Poincare-local},
	we obtain the following isoperimetric inequality.
	We recall that $\supp u:=\overline{\{x\in X: u(x)\neq 0\}}$
	denotes the support of $u$.
	Moreover, we denote by ${\bf 1}_E:X\to \{0,1\}$ the characteristic function of the set $E\subset X$. 
	
	\begin{corollary}
		There exists a constant $C_1\ge 1$ depending only on $s$ and the
		doubling constant such that for all measurable functions $u$ on $X$ with
		$\supp u \subset \overline{B(x,r)}$,
		where  $x\in X$ and  $0<r<\frac{1}{4}\diam X$, 
		\begin{align}\label{isoperimetric-func.}
			\|u\|_{L^1(B(x,r))} \leq C_1 r^s \|u\|_{\dot W^{s,1}(X)}.
		\end{align}
		In particular, for every set $E$ with  $\mu(E\setminus B(x,r))=0$, we have
		\begin{align}\label{isoperimetric-set}
			\mu(E)\leq C_1 r^s \|{\bf 1}_E\|_{\dot W^{s,1}(X)}.
		\end{align}
		Moreover, if we additionally assume that $(X,d,\mu)$ supports a $1$-Poincar\'e inequality, then the constant $C_1$ in \eqref{isoperimetric-func.}-\eqref{isoperimetric-set} can be chosen to be $C_1=(1-s)C$, where $C$ 
		depends only on the doubling and Poincar\'e constants. 
	\end{corollary}
	\begin{proof}
		We prove the result under the assumption of a $1$-Poincar\'e inequality.
		The statement under the mere doubling assumption then follows by applying \eqref{Poincare-local-2}
		instead of \eqref{Poincare-local}.
		By \cite[Lemma~3.7]{BB11} and connectedness of the space,
		there exists $0<\xi<1$ such that 
		$\mu(B(x,r))\leq \xi \cdot \mu(B(x,2r))$ for all $0<r<\frac{1}{4}\diam X$. Let $Q>s$ be a doubling dimension. Then we obtain from
		$\supp u \subset \overline{B(x,r)}$ and the H\"older inequality that
		\begin{align*}
			\dashint_{B(x,2r)}|u|\,d\mu
			&=\frac{\mu(B(x,r))}{\mu(B(x,2r))}\dashint_{B(x,r)}|u|\,d\mu\\
			&\leq \frac{\mu(B(x,r))}{\mu(B(x,2r))} \left(\dashint_{B(x,r)} |u|^{\frac{Q}{Q-s}}\,d\mu \right)^{\frac{Q-s}{Q}}\\
			&= \left(\frac{\mu(B(x,r))}{\mu(B(x,2r))} \right)^{\frac{s}{Q}} \left(\frac{1}{\mu(B(x,2r))}\int_{B(x,r)} |u|^{\frac{Q}{Q-s}}\,d\mu \right)^{\frac{Q-s}{Q}}\\
			&\leq \xi^{\frac{s}{Q}}\left(\dashint _{B(x,2r)}|u|^{\frac{Q}{Q-s}}\,d\mu \right)^{\frac{Q-s}{Q}}.
		\end{align*}
		By the triangle inequality for the $L^{\frac{Q}{Q-s}}$-norm,
		there is a constant $C>0$ depending only on the doubling and Poincar\'e constants such that
		\begin{align*}
			&\left(\dashint_{B(x,2r)} |u|^{\frac{Q}{Q-s}}\,d\mu \right)^{\frac{Q-s}{Q}}\\
			&\leq \left(\dashint_{B(x,2r)} |u-u_{B(x,2r)}|^{\frac{Q}{Q-s}}\,d\mu \right)^{\frac{Q-s}{Q}} +|u_{B(x,2r)}|\\
			&\overset{\eqref{Poincare-local}}{\le} C(1-s)\frac{r^s}{\mu(B(x, r))}\|u\|_{\dot W^{s,1}(X)}+ \xi^{\frac{s}{Q}}\left(\dashint _{B(x,2r)}|u|^{\frac{Q}{Q-s}}\,d\mu \right)^{\frac{Q-s}{Q}}.
		\end{align*}
		Subtracting the second term from both sides, we then obtain from the doubling property and the H\"older inequality that
		\begin{align*}
			\dashint_{B(x,r)}|u|\,d\mu
			&\leq \frac{\mu(B(x,2r))}{\mu(B(x,r))} \left(\dashint_{B(x,2r)} |u|^{\frac{Q}{Q-s}}\,d\mu \right)^{\frac{Q-s}{Q}}\\
			&\leq C_\mu (1-\xi^{\frac{s}{Q}})^{-1} C(1-s)\frac{r^s}{\mu(B(x,r))}\|u\|_{\dot W^{s,1}(X)}.
		\end{align*}
		This proves \eqref{isoperimetric-func.}. The estimate \eqref{isoperimetric-set} follows immediately by applying \eqref{isoperimetric-func.} to $u={\bf 1}_{E\cap B(x,r)}$.
	\end{proof}
	
	\subsection{Obstacle problem}\
	
	Given a nonempty  open set $\Omega\subset X$, a function $\psi:\Omega\to [-\infty,\infty)$,
	and a  function $f\in \dot W^{s,1}(X)$, we set
	\[
	\mathcal K_{\psi,f}(\Omega):=\{u\in \dot W^{s,1}(X):u\geq \psi \text{ a.e. on } \Omega \text{ and } u=f \text{ a.e. on } X\setminus \Omega \}.
	\]
	We say that $u\in \mathcal{K}_{\psi, f}(\Omega)$ is a \emph{solution of the $\mathcal{K}_{\psi,f}(\Omega)$-obstacle problem}
	(or briefly a \emph{$\mathcal K_{\psi,f}(\Omega)$-solution}) if 
	$$\|u\|_{\dot W^{s,1}(X)}\leq \|v\|_{\dot W^{s,1}(X)}\quad 
	\text{for all $v\in  \mathcal{K}_{\psi, f}(\Omega)$.}$$ 
	
	If the characteristic function of a set $E$ is a solution of the obstacle problem, we simply refer to $E$ as a solution. Moreover, when $\psi=\mathbf{1}_W$ for some $W\subset X$, we denote $\mathcal{K}_{W,f}:=\mathcal{K}_{\psi,f}$.
	
	We observe some useful properties of solutions.
	
	\begin{proposition}\label{prop2.4-22nov}
		Let 
		$x\in X$ and $0<r<\frac{1}{4}\diam X$.
		Suppose that $\mathcal K_{\psi,f}(\Omega)\neq \emptyset$ where $\Omega\subset B(x,r)$ 
		is a nonempty open set, $\psi:\Omega\to [-\infty,\infty)$, and $f\in \dot W^{s,1}(X)$.
		Then  there exists a $\mathcal K_{\psi,f}(\Omega)$-solution $u\in \dot W^{s,1}(X)$.
		
		Furthermore, such a $\mathcal K_{\psi, f}(\Omega)$-solution $u\in \dot W^{s,1}(X)$  is an $s$-superminimizer 
		{ in $\Omega$. }
		
		If moreover $f\in W^{s,1}(X)$, then such $u$ also belongs to $W^{s,1}(X)$.
	\end{proposition}
	\begin{proof}
		Since $\mathcal K_{\psi,f}(\Omega)\neq \emptyset$, there is
		a sequence $\{ u_i\}_{i\in\mathbb N}$ of functions $u_i\in \mathcal{K}_{\psi,f}(\Omega)$ such that
		\[
		\lim_{i\to \infty} \|u_i\|_{\dot W^{s,1}(X)}
		=I
		:=\inf\{\|u\|_{\dot W^{s,1}(X)}: u\in \mathcal{K}_{\psi,f}(\Omega)\}< \infty.
		\]
		We may assume that $\|u_i\|_{\dot W^{s,1}(X)}<I+1$.
		By  the isoperimetric inequality \eqref{isoperimetric-func.},
		we have for each $i\in \mathbb{N}$
		\begin{align}
			\int_X|u_i-f|\,d\mu
			&=\int_{\Omega}|u_i-f|\,d\mu\quad \textrm{since\ }u_i=f\ {\rm a.e.\ on\ }X\setminus \Omega\notag \\
			&\le
			{C_1r^s}
			\|u_i-f\|_{\dot W^{s,1}(X)}\notag \\
			&\leq
			{C_1r^s}
			\left(I+1+\|f\|_{\dot W^{s,1}(X)} \right), \label{eq2.7-16dec}
		\end{align}
		which implies that $u_i-f$ is  bounded in $L^1(X)$.
		Combining this with the fact that both $u_i$ and $f$ are bounded in $\dot W^{s,1}(X)$, we obtain that
		$u_i-f$ is a bounded sequence in $W^{s,1}(X)$. Thus by \cite[Theorem 5.2]{Kl25}, there exists a subsequence (not relabeled and denoted $\{u_i-f\}_{i\in\mathbb N}$) and a function $v\in W^{s,1}(X)$ such that $u_i-f\to v$ in $L^1\loc(X)$. We can select a further subsequence (still not relabeled and denoted $\{u_i-f\}_{i\in\mathbb N}$) such that $u_i(x)-f(x)\to v(x)$ for a.e. $x\in X$. Notice from $u_i\geq \psi$ a.e. on $\Omega$ and $u_i=f$ a.e. on $X\setminus\Omega$ that $v\geq \psi-f$ a.e. on $\Omega$ and $v=0$ a.e. on $X\setminus\Omega$. Letting $u:=v+f$, we then have that  $u\geq \psi $ a.e. on $\Omega$ and $u=f$ a.e. on $X\setminus \Omega$. 
		Moreover, $u_i\to u$ in $L^1\loc (X)$, and so  Fatou's lemma gives  that
		\[
		\|u\|_{\dot W^{s,1}(X)}\leq \liminf_{i\to \infty}\|u_i\|_{\dot W^{s,1}(X)}=I,
		\]
		which implies $u\in \mathcal K_{\psi,f}(\Omega)$ with minimal energy. Therefore, $u\in\dot W^{s,1}(X)$ is a $\mathcal K_{\psi,f}(\Omega)$-solution.
		
		We now prove the last two claims. First, 
		notice that $u+\varphi\in \mathcal K_{\psi,f}(\Omega)$ for every nonnegative function $\varphi\in W^{s,1}_0(\Omega)$, and hence by the definition of a $\mathcal K_{\psi,f}(\Omega)$-solution, we obtain  that $u\in \dot W^{s,1}(X)$ is an $s$-superminimizer 
		{in $\Omega$.}
		
		Moreover,
		if $f\in W^{s,1}(X)$, then  we have from \eqref{eq2.7-16dec} that the above sequence $u_i$ is bounded in $L^1(X)$  and $u_i\to u$ in $L^1_{\rm loc}$, and so Fatou's lemma gives that $u\in L^1(X)$. This implies that $u\in W^{s,1}(X)$ since $u\in \dot W^{s,1}(X)$.
		The proof is completed.
	\end{proof}
	By  Tonelli's theorem one can obtain a coarea formula, see e.g. \cite{KLLZ25} in the metric
	space setting and \cite{ Vi91}  in the Euclidean space: for a locally integrable function $u$
	on an open set $\Omega\subset X$,
	\begin{equation}\label{eq2.7-27nov}
		\|u\|_{\dot W^{s,1}(\Omega)}=\int_{-\infty}^\infty \| {\bf 1}_{\{ u>t\}}\|_{\dot W^{s,1}(\Omega)} \,dt.
	\end{equation}
	\begin{proposition}\label{prop2.5-22nov}		
		Let 
		$x\in X$ and $0<r<\frac{1}{4}\diam X$.
		Let $\Omega\subset B(x,r)$ be a nonempty open set, and let $A\subset \Omega$.
		Suppose that $\mathcal K_{A,0}(\Omega)\neq \emptyset$. 
		Then  there exists  a $\mathcal K_{A,0}(\Omega)$-solution $E$ with $A\subset E\subset \Omega$.
	\end{proposition}
	\begin{proof}
		By Proposition \ref{prop2.4-22nov}, 
		there is   a $\mathcal K_{A,0}(\Omega)$-solution $u\in  W^{s,1}(X)$. By the coarea
		formula \eqref{eq2.7-27nov},  there exists $t\in (0,1)$ such that
		$\|{\bf 1} _{\{u>t\}}\|_{\dot W^{s,1}(X)}\leq \|u\|_{\dot W^{s,1}(X)} $. For such $0<t<1$,
		letting $E:=\{u>t\}$, we have from $u\in \mathcal K_{A,0}(\Omega)$ that ${\bf 1}_E\geq {\bf 1}_A$
		a.e. on $\Omega$ and ${\bf 1}_E=0$ a.e.  on $\Omega^c$.
		Thus $E\in \mathcal K_{A,0}(\Omega)$, and so  $E$ is a $\mathcal K_{A,0}(\Omega)$-solution.
		After possibly modifying $E$ in a set of zero measure, we get $A\subset E\subset\Omega$.
	\end{proof}
	
	\subsection{Capacity, thinness, and Hausdorff measure}\label{subsection5.1}\
	
	As before, let $0<s<1$. We define the \emph{$(s,1)$-capacity} of $A\subset X$ by
	\[
	C_{s,1}(A):=\inf \|u\|_{W^{s,1}(X)},
	\]
	where the infimum is taken over all functions $u\in W^{s,1}(X)$  satisfying $u\equiv 1$ in a neighborhood of $A$. 
	
	We denote $A\Subset F$ if $\overline A$ is a compact subset of $F$, and we recall that
	$\supp u:=\overline{\{ x\in X: u(x)\neq 0\}}$ is the support of $u$.
	For a bounded set $A$ and a bounded open set $F$ with $A\Subset F$, we define the \emph{condenser $(s,1)$-capacity} by 
	\[
	\caps(A,F):=\inf  \|u\|_{\dot W^{s,1}(X)},
	\]
	where the infimum is taken over all  functions $u\in\dot W^{s,1}(X)$ such that
	$u\equiv 1$ in a neighbourhood of $A$ and $\supp u\subset F$. Such $u$ are called
	\emph{admissible} for $\caps(A,F)$.  For each given bounded open set $F$,  the definition
	gives that the condenser capacity $\caps(.,F)$ is an outer capacity in the sense that
	for any $A\Subset F$, 
	\begin{equation}\label{eq2.8-22nov}
		\caps(A,F)=\inf\{ {\rm cap}_{s,1}(G,F): A\subset G\Subset F,\ G \text{\ is open} \}.
	\end{equation}
	Recall that $(X,d,\mu)$ is a complete, connected metric measure space equipped with a doubling measure.  Then we  have from \cite[Proposition 4.1]{BB23} that for all balls $B(x,r)\subset B(x,R)\subsetneq X$ with $\theta R \leq 2r\leq R$ for a given $0<\theta<1$, 
	\begin{equation}\label{eq4.1-28oct}
		\frac{1}{C_2}\frac{\mu(B(x,r))}{r^s}\le\caps (B(x,r), B(x,R))\le C_2 \frac{\mu(B(x,r))}{r^s},
	\end{equation}
	where $C_2$ depends only on $\theta$, $s$ and the doubling constant.

	\begin{lemma}\label{lem4.3-29oct}
		
		Let $x\in X$ and $0<r<r_1<r_2<\frac{1}{4}\diam X$.
		Then there is a constant $C>0$  depending only on $s$ and the doubling constant such that
		for every $A\subset B(x,r)$,
		\begin{align}
			\caps(A, B(x,r_2))\leq&\, \caps(A,B(x,r_1))\label{e4.2-29oct} \\
			\leq&\,  C\left(1+\frac{r_2^s}{(r_1-r)^s}\right)\caps(A, B(x,r_2)). \label{e4.3-29oct}
		\end{align} 
	\end{lemma}
	\begin{proof}
		We may assume that $A\neq\emptyset$.
		The first inequality is obtained by the definition of the condenser $(s,1)$-capacity.
		We will prove the second inequality. Let $u$ be an arbitrary admissible function for
		$\caps(A, B(x,r_2))$. Let $\eta$ be a $\frac{2}{r_1-r}$-Lipschitz function such that
		$\eta\equiv 1$ in a neighbourhood of $A$, $\supp \eta\Subset B(x,r_1)$ and $0\leq \eta\leq 1$.
		Then $h:=u\eta$ is admissible for $\caps(A, B(x, r_1))$.
		Let $A_j:=B_{j+1}\setminus B_{j}$ where $B_j:= B(y, 2^{j} (r_1-r))$ and $j\in\mathbb Z$ for
		$y\in X$.
		We have from the doubling property that for $y\in X$,
		\begin{align}
			&\quad \int_{X}  \frac{  |\eta(y)-\eta(x)|}{\mu(B_{x,y}) d(x,y)^s}  \,d\mu(x)\notag \\
			&= \ \sum_{j=-\infty}^0 \int_{A_j} \frac{  |\eta(y)-\eta(x)|}{\mu(B_{x,y}) d(x,y)^s}  \,d\mu(x) + \sum_{j=1}^\infty \int_{A_j} \frac{  |\eta(y)-\eta(x)|}{\mu(B_{x,y}) d(x,y)^s}  \,d\mu(x) \notag \\
			&\leq   \sum_{j=-\infty}^0 \int_{A_j} \frac{ \frac{2}{r_1-r}d(x,y) }{\mu(B_j) d(x,y)^s}  \,d\mu(x) +  \sum_{j=1}^\infty \int_{A_j} \frac{2}{\mu(B_j) d(x,y)^s}  \,d\mu(x) \notag \\
			&\simeq    \frac{1}{r_1-r}  \sum_{j=-\infty}^0 (2^{j}(r_1-r))^{1-s} +  \sum_{j=0}^\infty (2^j(r_1-r))^{-s}\notag \\
			&\simeq    \frac{1}{(r_1-r)^s} \left( \frac{1}{1-2^{-(1-s)}} + \frac{1}{1-2^{-s}} \right)\notag \\
			&\lesssim \frac{1}{s(1-s)}\frac{1}{(r_1-r)^s}, \label{eq2.6-29oct}
		\end{align}
		where the third line is obtained  since $\eta$ is $\frac{2}{r_1-r}$-Lipschitz and
		$0\leq \eta\leq 1$, and the fourth line is obtained since the radius of $B_j$ is $2^j(r_1-r)$. 
		Hence 
		\begin{align}
			&\quad \caps(A, B(x, r_1))\notag \\
			&\leq   \int_X\int_X \frac{|h(x)-h(y)|}{\mu(B_{x,y})d(x,y)^s}\,d\mu(x)\,d\mu(y)\notag \\
			&\leq   \int_X\int_X \frac{|u(x)-u(y)|\eta(x)}{\mu(B_{x,y})d(x,y)^s}\,d\mu(x)\,d\mu(y)\notag \\
			&\quad + \int_X\int_X\frac{|u(y)||\eta(x)-\eta(y)|}{\mu(B_{x,y})d(x,y)^s}\,d\mu(x)\,d\mu(y)\notag \\
			&\lesssim  \int_X\int_X \frac{|u(x)-u(y)|}{\mu(B_{x,y})d(x,y)^s}\,d\mu(x)\,d\mu(y)\notag \\ 
			&\quad +\frac{1}{s(1-s)} \frac{1}{(r_1-r)^s }\int_{B(x,r_2)} {|u(y)|} \,d\mu(y) \notag\\
			&\overset{\eqref{isoperimetric-func.}}{\lesssim_s}
			\int_{X} \int_{X}\frac{|u(x)-u(y)|}{\mu(B_{x,y})d(x,y)^s} \,d\mu(x)  \,d\mu(y)\notag\\
			&\, +\frac{r_2^s}{(r_1-r)^s}\int_{X} \int_{X}\frac{|u(x)-u(y)|}{\mu(B_{x,y})d(x,y)^s} \,d\mu(x)  \,d\mu(y). \label{equ4.2-29oct}
		\end{align}
		Here the comparison constant in the last estimate depends only on $s$ and the doubling constant.
		Since $u$ is an arbitrary admissible function for $\caps(A, B(x, r_2))$, this gives the desired inequality.
	\end{proof}
	\begin{lemma}\label{lemma4.6}

		Let 
		$x\in X$ and $0<r<R<\frac{1}{4}\diam X$.
		Let   $A\subset B(x,r)$.
		Then there exists a $\mathcal{K}_{A,0}(B(x,R))$-solution $E$ satisfying $A\subset E\subset B(x,R)$ and
		\[
		\|{\bf 1}_{E}\|_{\dot W^{s,1}(X)}\leq \caps(A,B(x,R)).
		\]
		Moreover, every $\mathcal{K}_{A,0}(B(x,R))$-solution $F$ satisfies
		$A\subset F\subset B(x,R)$ (after possibly modifying in a set of zero measure) and
		\begin{equation}\label{eq2.7-19nov}
			\|{\bf 1}_{F}\|_{\dot W^{s,1}(X)}\leq \caps(A,B(x,R)).
		\end{equation}
	\end{lemma}
	\begin{proof}
		Let $\varepsilon>0$ be arbitrary. 
		We have that $\caps(A,B(x,R))<\infty$ because
		\begin{align*}
			\caps(A,B(x,R)) & \overset{\eqref{e4.2-29oct} \& \eqref{e4.3-29oct}}{\lesssim_{r,R,s}} \caps (A, B(x,2r))\\
			&\leq \caps (B(x,r), B(x,2r)) \overset{\eqref{eq4.1-28oct}}{<}\infty.
		\end{align*}
		By the definition of the condenser capacity, we find $u\in \dot W^{s,1}(X)$ with $u= 1$ in a neighbourhood of $A$, $\supp u \Subset  B(x,R)$, and
		\[
		\caps(A,B(x,R))+\varepsilon\ge \|u\|_{\dot W^{s,1}(X)}.
		\]
		Notice that such $u\in \mathcal{K}_{A,0}(B(x,R))$, and so
		$\mathcal K_{A,0}(B(x,R))\neq\emptyset$. 
		By Proposition \ref{prop2.5-22nov}, 
		there exists a $\mathcal{K}_{A,0}(B(x,R))$-solution $E\subset X$. By modifying in a set of zero measure, we may assume that $A\subset E\subset B(x,R)$.
		Combining this with the fact that $u\in \mathcal{K}_{A,0}(\Omega)$, the above inequality gives that 
		$\|{\bf 1}_{E}\|_{\dot W^{s,1}(X)}\leq \caps(A,B(x,R)) +\varepsilon$. Letting $\varepsilon\to0$, this gives the main claim.
		The last claim \eqref{eq2.7-19nov} is obtained by the main claim and the fact that $\|{\bf 1}_{F}\|_{\dot W^{s,1}(X)}\leq \|{\bf 1}_{E}\|_{\dot W^{s,1}(X)}$ because $F$ is a $\mathcal{K}_{A,0}(B(x,R))$-solution and ${\bf 1}_E\in \mathcal K_{A,0}(B(x,R))$.
	\end{proof}

	For $A\subset X$ and a point $x\in X$,
	$A$ is said to be \emph{$s$-thin} at $x$ if
	\begin{equation}\label{eq2.8-2011}
		\lim_{r\to0^+} \frac{\caps(A\cap B(x,r), B(x,2r))}{\caps(B(x,r),B(x,2r))} =0.
	\end{equation}
	If $A$ is not $s$-thin at $x$, we say that it is \emph{$s$-thick} at $x$. 
	
	\begin{lemma}\label{lem4.4-29oct}

		Let $x\in X, R>0$ and let $A\subset X$ and $M>1$ be such that
		\begin{equation}\label{eq4.6-29oct}
			\lim_{i\to\infty}\frac{\caps(A\cap B(x,M^{-i}R), B(x, 2M^{-i}R))}{\caps(B(x,M^{-i}R), B(x, 2M^{-i}R))}=0.
		\end{equation}
		Then $A$ is $s$-thin at $x$, i.e. 
		\[
		\lim_{r\to0^+}  \frac{\caps(A\cap B(x,r), B(x,2r))}{\caps(B(x,r),B(x,2r))} =0.
		\]
	\end{lemma}
	\begin{proof}
		For $r\in (M^{-i-1}R,\, M^{-i}R]$ with $2M^{-i}R<\frac{1}{4}\diam X$, applying \eqref{eq4.1-28oct} and \eqref{e4.3-29oct} 
		{to}
		$r_1=2r,\, r_2=2M^{-i}R$ yields
		\[
		\frac{\caps(A\cap B(x,r),\, B(x,2r))}{\caps(B(x,r),\, B(x,2r))}\lesssim_{M,s}
		\frac{\caps(A\cap B(x,r),\, B(x,2M^{-i}R))}{\caps(B(x,M^{-i}R),\, B(x,2M^{-i}R))}.
		\]
		Since $B(x,r)\subset B(x,M^{-i}R)$, we obtain that
		\[
		\frac{\caps(A\cap B(x,r),\, B(x,2r))}{\caps(B(x,r),\, B(x,2r))}\lesssim_{M,s}
		\frac{\caps(A\cap B(x,M^{-i}R),\, B(x,2M^{-i}R))}{\caps(B(x,M^{-i}R),\, B(x,2M^{-i}R))}.
		\]
		By \eqref{eq4.6-29oct}, the right-hand side tends to $0$ as $i\to\infty$. This completes the proof.
	\end{proof}

	\begin{lemma}\label{openset}
		
		Let  $x\in X\setminus A$ be such that  $A$ is $s$-thin at $x$. Then there exists an open set $W$ such that $A\subset W$ and  $W$ is $s$-thin at $x$.

	\end{lemma}
	\begin{proof}
		Let $B_i:=B(x, 2^{-i})$, $i\in\mathbb N$.
		By Lemma \ref{lem4.3-29oct}, we have
		\[
		\caps(A\cap \overline{B_i}, 2B_i)\lesssim_s  \caps(A\cap\overline{B_i}, 4B_i)\leq  \caps(A\cap {2B_i}, 4B_i).
		\]
		Recall that $\caps$ is an outer capacity in the sense \eqref{eq2.8-22nov}. Hence for each $i\in\mathbb N$, there is an open set $W_i$ such that $A\cap \overline{B_i}\subset W_i\Subset 2B_i$ and 
		\[
		{\caps(W_i, 2B_i)}\leq { \caps(A\cap \overline{B_i}, 2B_i)}+ \frac{{\caps(B_i, 2B_i)}}{i}.
		\]
		Set
		\[
		W:= (X\setminus \overline{B_1}) \cup\left( \bigcup_{i=2}^\infty \left(\bigcap_{j=1}^{i-1}W_j\right)\setminus \overline{B_{i}}\right).
		\]
		Then $W$ is open since $X\setminus \overline{B_1}$ and each $(\bigcap_{j=1}^{i-1}W_j)\setminus \overline{B_{i}}$ is open. Moreover $A\subset W$ and $W\cap B_i\subset W_i$ for each $i\in\mathbb N$. Hence from the two above estimates, we obtain
		\begin{align*}
			\frac{\caps(W\cap B_i, 2B_i)}{\caps(B_i, 2B_i)}
			&\leq \frac{\caps(W_i, 2B_i)}{\caps(B_i, 2B_i)}\\
			&\leq \frac{ \caps(A\cap \overline{B_i}, 2B_i)}{\caps(B_i, 2B_i)} + \frac{1}{i}\\
			&\lesssim_s \frac{ \caps(A\cap 2B_i, 4B_i)}{\caps(B_i, 2B_i)} + \frac{1}{i}\\ 
			&\overset{\eqref{eq4.1-28oct}}{\simeq_s}\frac{ \caps(A\cap 2B_i, 4B_i)}{\caps(2B_i, 4B_i)} + \frac{1}{i} \\
			&\to 0 \quad \text{as $i\to\infty$},
		\end{align*}
		since $A$ is $s$-thin at $x$. By Lemma \ref{lem4.4-29oct}, we conclude that $W$ is $s$-thin at $x$.

	\end{proof}
	
	By \cite[Lemma 4.9]{BB23}, if $(X,d,\mu)$ is a complete, connected metric measure space equipped
	with a doubling measure, then for
	$A\subset B(x,r)$ with
	$B(x,3r)\subsetneq X$, we have
	\begin{equation}\label{eq4.7-29oct}
		\frac{C_{s,1}(A)}{1+r^{s}}\lesssim_s \caps(A, B(x, 2r))\lesssim_s \left( 1+\frac{1}{r^{s}}\right)C_{s,1}(A),
	\end{equation}
	where the comparison constant depends only on $s$ and the doubling constant.

	\begin{example}\label{ex3.9-15dec}
		As usual, let $0<s<1$. We show that there exists a complete, connected metric measure space $(X,d,\mu)$ equipped a doubling measure, and a set $A$ that is $s$-thin at
		$O\in X$, such that $O\in \overline{A}\setminus A$ and $C_{s,1}(\{O\})>0$.
		Let $X=\mathbb{R}^n$ with the Euclidean distance $d$ and define the measure
		\[
		d\mu(x): =w(x)dx := |x|^{\delta}\,dx
		\]
		where $  -n < \delta < s-n.$
		Then $(X,d,\mu)$ is a metric measure space equipped a doubling measure {by \cite[Corollary 15.35]{HKM06} applied with $-n<\delta$.}
		For the origin $O$ and  { for all $ r>0$,}
		\begin{equation}\label{eq2.22-13mar}
			\mu(B(O,r)) = \int_{B(O,r)} |x|^{\delta}\,dx \simeq r^{\,n+\delta}.
		\end{equation}
		By \cite[Theorem 1.3]{BB23} { applied with \eqref{eq2.22-13mar}},  it follows from $0<n+\delta<s$ that $C_{s,1}(\{O\})>0$.			
		For any point $x\neq O$, the weight $|x|^{\delta}$ is bounded in a neighborhood of $x$. Let $r_j:= d(O,x)/(j+2)$ and $B_j:= B(x,r_j)$. Then we have that
		\[
		C_{s,1}(\{x\}) \overset{\eqref{eq4.7-29oct}}{\lesssim_s} \text{cap}_{s,1}(B_j, 2B_j) 
		\overset{\eqref{eq4.1-28oct}}{\simeq_s} \frac{\mu(B_j)}{r_j^s} 
		\lesssim |x|^\delta r_j^{\,n-s} \to 0 \quad \text{as } j\to\infty.
		\]
		Hence $C_{s,1}(\{x\}) = 0$ for every $x\in \mathbb{R}^n \setminus \{O\}$.
		Now let $\{x_j\}_{j\in\mathbb{N}}$ be a sequence in $\mathbb{R}^n\setminus \{O\}$ such that $x_j \to O$ as $j\to \infty$, and define
		\[
		A = \bigcup_{j\in\mathbb{N}} \{x_j\}.
		\]
		Then clearly $O \in \overline{A} \setminus A$.  
		Using subadditivity of capacity and the above estimate for $C_{s,1}(\{x_j\})$, we obtain that
		\begin{align*}
			&\quad \lim_{r\to0^+} \frac{\caps(A\cap B(O,r), 2B(O,r))}{\caps(B(O,r), 2B(O,r))}\\ 
			&\overset{\eqref{eq4.7-29oct}}{\lesssim_s} \lim_{r\to0^+} 
			\sum_{j\in\mathbb N:\, x_j\in B(O,r)}\frac{\left( 1+
				\frac{1}{r^s}\right)C_{s,1}(\{x_j\})}{\caps(B(O,r), 2B(O,r))} \equiv 0.
		\end{align*} 
		This shows that $A$ is $s$-thin at $O$. {The claim of Example \ref{ex3.9-15dec} follows.}
	\end{example}

	The measure-theoretic boundary $\partial^*E$ of a set $E\subset X$ is the set of points $x\in X$ at which both $E$ and its complement have strictly positive upper density, i.e.
	\begin{equation}\label{eq2.19-2011}
		\limsup_{r\to 0^+} \frac{\mu(B(x,r)\cap E)}{\mu(B(x,r))}>0 \quad \text{and }\quad 
		\limsup_{r\to 0^+} \frac{\mu(B(x,r)\setminus E)}{\mu(B(x,r))}>0.
	\end{equation}
	The measure-theoretic interior and exterior of $E$ are respectively defined  by
	\begin{equation}\label{eq2.20-2011}
		I_E:=\left\{x\in X:\lim_{r\to 0^+} \frac{\mu(B(x,r)\setminus E)}{\mu(B(x,r))}=0 \right\}
	\end{equation}
	and
	\begin{equation}\notag
		O_E:=\left\{x\in X:\lim_{r\to 0^+} \frac{\mu(B(x,r)\cap E)}{\mu(B(x,r))}=0 \right\}.
	\end{equation}
	Then  $X=\partial^*E \cup I_E\cup O_E$.
	
	\begin{lemma}\label{lem4.10}

		Let 
		$x\in X$ and $0<r<\frac{1}{4}\diam X$.
		Let  $A\subset B(x,r)$ be such that  $x\in I_A$.
		Then there exists $\sigma_r\in (0,r]$ such that
		\[
		\frac{\mu(B(x,\sigma_r))}{\sigma_r^s}\lesssim \, \caps(A,B(x,2r)).
		\]
	\end{lemma}
	\begin{proof}
		By Lemma \ref{lemma4.6} we find a set $E\subset B(x,2r)$ such that $A\subset E$ and 
		\begin{align}\label{lem4.10-1}
			\|{\bf 1}_E\|_{\dot W^{s,1}(X)}\leq \caps(A,B(x,2r)).
		\end{align}
		{By \cite[Lemma 3.7]{BB11},}
		there exists $0<\beta<1$ depending only on the doubling constant such that $ \mu(B(x,2r))\leq \beta \mu(B(x,4r))$. In particular, since $E\subset B(x,2r)$,
		\begin{equation}\label{eq2.19-22nov}
			\mu(E\cap B(x,4r))\le \beta \mu(B(x,4r)).
		\end{equation}
		Since $x\in I_A\subset I_E$, there is $0<r_0<r$ such that for all $0<t<r_0$,
		\begin{equation}\label{eq2.20-22nov}
			\mu(E\cap B(x,t))>  \beta \mu(B(x,t)).
		\end{equation}
		Setting
		\[
		\sigma':=\inf\{r'>0:\mu(E\cap B(x,r'))\le \beta\mu(B(x,r'))\},
		\]
		we have by \eqref{eq2.19-22nov}-\eqref{eq2.20-22nov} that $r_0\leq\sigma'\le 4r$.
		By the definition of $\sigma'$, there is $\sigma$
		with $\sigma'\leq \sigma< 3 \sigma'/2$ 
		such that $\mu(E\cap B(x,\sigma))\le \beta\mu(B(x,\sigma))$.
		Moreover, since $\sigma/2<3\sigma'/4$, we have from the definition of $\sigma'$ that
		$\mu(E\cap B(x,\sigma/2))>  \beta \mu(B(x,\sigma/2))$.
		Using the doubling property,
		we obtain that
		\begin{align}
			\beta\frac{1}{C_{\mu}}\leq \frac{\mu(E\cap B(x,\sigma/2))}{C_{\mu}\mu(B(x,\sigma/2))} \leq \frac{\mu(E\cap B(x,\sigma))}{\mu(B(x,\sigma))}
			\leq \beta.   \label{lem4.10-2}
		\end{align}
		By the fractional Poincar\'e inequality \eqref{Poincare-local-1}, we have that
		\begin{align}
			&\, \sigma^s\|{\bf 1}_E\|_{\dot W^{s,1}(X)}\notag \\
			&\geq \sigma^s\int_{B(x, \sigma)}\int_{B(x, \sigma)}\frac{|{\bf 1}_E(x)-{\bf 1}_E(y)|}{d(x,y)\mu(B_{x,y})}\,d\mu(x)\,d\mu(y)\notag\\
			&\overset{\eqref{Poincare-local-1}}{\gtrsim} \int_{B(x,\sigma)}|{\bf 1}_E-({\bf 1}_{E})_{B(x,\sigma)}|\,d\mu\notag\\
			&=\frac{2\mu(E\cap B(x,\sigma))\mu(B(x,\sigma)\setminus E)}{\mu(B(x,\sigma))}\notag\\
			&\overset{\eqref{lem4.10-2}}{\simeq_\beta} \mu(B(x,\sigma)).\label{lem4.8-3}
		\end{align}
		Thus 
		\begin{align*}
			\caps(A,B(x,2r))\overset{\eqref{lem4.10-1}}{\geq} \|{\bf 1}_E\|_{\dot W^{s,1}(X)}\gtrsim_{\beta} \frac{\mu(B(x,\sigma))}{\sigma^s}\simeq \frac{\mu(B(x,\sigma/6))}{(\sigma/6)^s}.
		\end{align*}
		Let $\sigma_r:=\sigma/6<r$. Since $\beta$ depends only on the doubling constant, this gives the desired claim.
	\end{proof}
	
	We record the following simple lemma.   
	\begin{lemma}[{\cite[Lemma 4.7]{La20}}] \label{lem4.11}
		Let $f:(0,\infty)\to (0,\infty)$ be such that $\lim_{r\to 0^+}f(r)=0$.
		Pick $\sigma_r\in (0,r]$ for every $r>0$.
		Then 
		\[
		\limsup_{r\to 0^+}\frac{f(r)}{f(\sigma_r)}\ge \frac{1}{2}.
		\]
	\end{lemma}
	
	\begin{lemma}\label{lem2.12-15dec}

		Let $x\in X$ with $C_{s,1}(\{ x\})>0$. Then $\{x\}$ is $s$-thick at $x$.
	\end{lemma}
	\begin{proof}
		If $\limsup_{r\to0^+}\left(\caps(B(x,r), B(x,2r))\right)^{-1}=:c>0$, then
		\begin{align*}
			\limsup_{r\to0^+}\frac{\caps(\{x\}, B(x,2r))}{\caps(B(x,r),B(x,2r))}&\overset{\eqref{eq4.7-29oct}}{\gtrsim} \limsup_{r\to0^+}\frac{{C}_{s,1}(\{x\} )}{\caps(B(x,r),B(x,2r))}\\
			&=\, c \cdot {C_{s,1}(\{x\})}>0,
		\end{align*}
		which implies that $\{x\}$ is $s$-thick at $x$. Thus we may assume that 
		\begin{equation}\label{eq4.8-29oct}
			\limsup_{r\to0^+}\frac{1}{\caps(B(x,r),B(x,2r))}=0.
		\end{equation}
		Let $r>0$ {be} sufficiently small that
		\[
		\frac{1}{\caps(B(x,r),B(x,2r))}\leq 1.
		\]
		Recall from \eqref{eq2.8-22nov} that the condenser $(s,1)$-capacity is an outer capacity, and so there exists a sufficiently small $0<t<r$ such that 
		\begin{equation}\label{eq2.24-22nov}
			\caps(\{x\}, B(x,2r)) \geq \caps (B(x,t), B(x,2r)) -1.
		\end{equation}
		By Lemma \ref{lem4.10}, we find $\sigma_r\in(0,r]$ such that
		\begin{align}\notag
			\caps(B(x,t),B(x,2r))
			&\gtrsim   \frac{\mu(B(x,\sigma_r))}{\sigma_r^s }\\
			&\overset{\eqref{eq4.1-28oct}}{\simeq_s}   \caps(B(x,\sigma_r),B(x,2\sigma_r)).\label{eq2.25-22nov}
		\end{align}
		It follows that
		\begin{align*}
			&\frac{\caps(\{x\}, B(x,2r))} {\caps(B(x,r),B(x,2r))}\\
			\overset{\eqref{eq2.24-22nov}}{\geq} &\, \frac{{\rm cap}_{s,1}(B(x,t), B(x,2r))}{{\caps(B(x,r),B(x,2r))}}- \frac{1}{\caps(B(x,r),B(x,2r))}\\
			\overset{\eqref{eq2.25-22nov}}{\gtrsim}_s&\,  \frac{\caps (B(x,\sigma_r), B(x,2\sigma_r)) }{\caps(B(x,r),B(x,2r))}-  \frac{1}{\caps(B(x,r),B(x,2r))}.
		\end{align*}
		Notice from Lemma \ref{lem4.11} applied to $f(r):=\left(\caps(B(x,r),B(x,2r)) \right)^{-1}$ that
		\[
		\limsup_{r\to0^+} \frac{\caps (B(x,\sigma_r), B(x,2\sigma_r)) }{\caps(B(x,r),B(x,2r))}\geq \frac{1}{2}.
		\]
		Combining this with \eqref{eq4.8-29oct}, we conclude from the  above estimates that 
		\[
		\limsup_{r\to 0^+}\frac{\caps(\{x\},B(x,2r))}{\caps(B(x,r),B(x,2r))}>0,
		\]
		which implies that $\{x\}$ is $s$-thick at $x$.
	\end{proof}
	
	For any $A\subset X$ and $0<R<\infty$, the
	\emph{$R$-restricted Hausdorff content of codimension $s$} of $A$ is defined by
	\[
	\mathcal{H}^{-s}_R(A):=\inf{\left\{\sum_{i} \frac{\mu(B(x_i,r_i))}{r_i^s}: A\subset \bigcup_{i} B(x_i,r_i),\, r_i\leq R \right\}},
	\]
	where the infimum is taken over finite and countable collections of balls.
	The \emph{Hausdorff measure  of codimension $s$} of $A$ is then defined by
	\[
	\mathcal{H}^{-s}(A):=\lim_{R\to 0^+} \mathcal{H}^{-s}_R(A).
	\]
	
	\begin{lemma}\label{lem5.7-19nov}
		
		For any $E\subset X$ with ${\bf 1}_E\in \dot W^{s,1}(X)$, we have
		$\mathcal H^{-s}(\partial^*E)=0$, and for every ball $B$,
		\[
		{\rm cap}_{s,1}(\partial^*E\cap B, 2B)=0.
		\]
	\end{lemma}
	Here $\partial^*E$ is defined as in \eqref{eq2.19-2011}.
	\begin{proof}
		By the 5-covering lemma, there exists a sequence of pairwise disjoint unit balls $B(x_i,1)$ so that  $\bigcup_{i\in\mathbb N} B(x_i,5)=X$.
		By \cite[Lemma 7.1]{BB23}, we have that for each $i\in\mathbb N$, there is a closed
		set $X_i$ such that $B(x_i,5)\subset X_i\subset \overline{B(x_i,10)}$,  $\mu{\lfloor}_{X_i}$ is doubling on $(X_i, d,\mu{\lfloor}_{X_i})$,
		and 
		{$\mu{\lfloor}_{X_i}(B(x,r))\simeq \mu(B(x,r))$}
		for all $x\in X_i$ 
		and
		$0<r<2\diam X_i$, where the comparison constant depends only on the doubling constant.
		Recall that $\overline{B(x_i,10)}$ is compact since any complete doubling metric measure space is proper by \cite[Lemma 4.1.14]{HKST15}.
		Since a closed subset of a compact set is compact,
		each $(X_i,d)$ is compact.
		
		Let $E$ satisfy ${\bf 1}_E\in \dot W^{s,1}(X)$.
		By \cite[Theorem 1.2]{Kl25} applied in the compact doubling metric measure spaces $({X_i},d,\mu{\lfloor}_{{X_i}})$, we obtain from ${\bf 1}_E\in  W^{s,1}(X_i, d, \mu{\lfloor}_{X_i})$ and $\mu{\lfloor}_{X_i}(B(x,r))\simeq\mu(B(x,r))$ for $x\in X_i$ and $0<r<2\diam X_i$ that $\mathcal{H}^{-s}(\partial^* E\cap X_i)=0$ for all $i$.
		Because $\bigcup_{i\in\mathbb N}{X_i}=X$ and since the Hausdorff measure of codimension $s$ is subadditive, we conclude that $\mathcal H^{-s}(\partial^*E)=0$.
		
		Now let $x\in X$, $r>0$ and $\varepsilon\in (0,r/3)$. For any cover $\{B(x_j, r_j)\}_{j}$
		(finite or countable)
		of $\partial^*E\cap B(x,r)$ with $r_j\leq \varepsilon$, the subadditivity property of
		$\caps$ (see e.g. \cite[Proposition 3.4]{BB23}) implies that
		\begin{align*}
			&\quad \caps(\partial^*E\cap B(x,r), B(x,2r))\\
			&\leq   \sum_{j:\, B(x_j,r_j)\cap B(x,r)\neq\emptyset} \caps(B(x_j,r_j), B(x,2r)), \\
			&\leq  \sum_{j:\, B(x_j,r_j)\cap B(x,r)\neq\emptyset} \caps(B(x_j,r_j), B(x_j,2r_j))\\
			&\overset{\eqref{eq4.1-28oct}}{\simeq_s} \sum_{j:\, B(x_j,r_j)\cap B(x,r)\neq\emptyset} \frac{\mu(B(x_j,r_j))}{r_j^s},
		\end{align*}
		where the third line holds since $B(x_j,2r_j)\subset B(x, r+ 3r_j) \subset B(x, r+3\varepsilon) \subset B(x,2r)$ for all $j\in\mathbb N$ with $B(x_j,r_j)\cap B(x,r)\neq\emptyset$.
		Therefore,
		\[
		\caps(\partial^*E\cap B(x,r), B(x,2r))\lesssim_s \mathcal H_\varepsilon^{-s}(\partial^* E\cap B(x,r))\leq \mathcal H_\varepsilon^{-s}(\partial^* E).
		\]
		Since $\varepsilon\in (0,r/3)$ is arbitrary and $\mathcal H^{-s}(\partial^*E)=0$, the claim follows.
	\end{proof}
	
	For $a\in \mathbb{R}$, we denote the $\lceil a \rceil$ the smallest integer at least $a$.
	
	\begin{lemma}\label{lemma4.11-Nov.12}

		Let $x\in X$, $r>0$, and let $E\subset X$ such that 
		\begin{align}\label{lem4.11-1}
			\frac{\mu(E\cap B(x,2r))}{\mu(B(x,2r))}\leq \frac{1}{2C_{\mu}^{\lceil \log_2 50 \rceil}}.
		\end{align}
		Then 
		\begin{equation}\label{eq2.27-22nov}
			\caps((I_E\cup \partial^*E)\cap B(x,r),B(x,2r))\lesssim_s\|{\bf 1}_E\|_{\dot W^{s,1}(X)}.
		\end{equation}
	\end{lemma}
	Here  $\partial^*E, I_E$ are defined as in \eqref{eq2.19-2011}-\eqref{eq2.20-2011}.
	\begin{proof}
		We may assume that ${\bf 1}_E\in \dot W^{s,1}(X)$ because otherwise the claim is obvious.
		By Lemma \ref{lem5.7-19nov}, it suffices to prove that
		\begin{align}\label{lem4.11-2}
			\caps(I_E\cap B(x,r),B(x,2r))\lesssim\|{\bf 1}_E\|_{\dot W^{s,1}(X)}.
		\end{align}
		Let $y\in I_E\cap B(x,r)$. Since $y\in I_E$, there exists $t'\in(0,\frac{r}{11})$
		such that for every $0<t<t'$,
		\begin{equation}\label{eq5.13-19nov}
			\frac{\mu(E\cap B(y,t))}{\mu(B(y,t))}>\frac{1}{2}.
		\end{equation}
		Since $y\in B(x,r)$, we have that $B(y,\frac{r}{11})\subset B(x,2r)$ so that
		in particular $\mu(E\cap B(y,\frac{r}{11}))\leq \mu(E\cap B(x,2r))$,
		and hence the doubling property gives 
		\begin{equation}
			\frac{\mu(E\cap B(y,\frac{r}{11}))}{\mu(B(y,\frac{r}{11}))}
			\leq C_{\mu}^{\lceil \log_2 50 \rceil} \frac{\mu(E\cap B(x,2r))}{ \mu(B(x,2r))}  \overset{\eqref{lem4.11-1}}{\leq}\frac{1}{2}.\label{eq5.14-19nov}
		\end{equation}
		Setting 
		$$
		t_y':=\inf\{t>0:\mu(E\cap B(y,t))\le\frac{1}{2}\mu(B(y,t))\},
		$$
		it follows from \eqref{eq5.13-19nov}-\eqref{eq5.14-19nov}
		that $t'\le t_y'\le\frac{r}{11}$. 
		By the definition of $t'_y$, there is $t_y$ with $t'_y\leq t_y<\tfrac{11}{10}t'_y$ such that
		$\mu(E\cap B(y,t_y))\le\frac{1}{2}\mu(B(y,t_y))$.
		Moreover, since $t_y/2<t'_y$, the definition of $t_y'$ gives that
		$\mu(E\cap B(y,t_y/2))>\frac{1}{2}\mu(B(y,t_y/2))$.
		Using the doubling property
		we obtain that
		\begin{align}
			\frac{1}{2C_{\mu}}<\frac{\mu(E\cap B(y,t_y/2))}{C_{\mu}\mu(B(y,t_y/2))} \leq \frac{\mu(E\cap B(y,t_y))}{\mu(B(y,t_y))}\leq \frac{1}{2}.
		\end{align}
		By an argument similar to \eqref{lem4.8-3}, we get
		\begin{align}\label{lem4.11-3}
			\mu(B(y,t_y))&\lesssim t_y^s\int_{B(y, t_y)}\int_{B(y, t_y)}\frac{|{\bf 1}_E(x)-{\bf 1}_E(y)|}{\mu(B_{x,y})d(x,y)^s}\,d\mu(x)\,d\mu(y)
		\end{align}
		for each $y\in I_E\cap B(x,r)$ and for some $0<t_y<\frac{r}{10}$.
		By the 5-covering lemma, we obtain a sequence $\{B(y_i, t_i)\}_i$ of pairwise disjoint balls with $y_i\in I_E\cap B(x,r)$ and $0<t_i<\frac{r}{10}$, where each ball $B(y_i,t_i)$ satisfies \eqref{lem4.11-3}, such that 
		\[
		I_E\cap B(x,r)\subset \bigcup_{i\in\mathbb N} B(y_i,5 t_i).
		\]
		Then by the  subadditivity of capacity, see e.g. \cite[Proposition 3.4]{BB23}, we have
		\begin{align*}
			&\quad \caps(I_E\cap B(x,r),B(x,2r))\\
			&\leq \sum_{i\in \mathbb N} \caps(B(y_i,5 t_i), B(x,2r))\\
			&{\leq} \sum_{i\in \mathbb N} \caps(B(y_i,5 t_i), B(y_i,10  t_i))
			\quad \textrm{since\ }B(y_i,10 t_i)\subset B(x,2r)\\
			&\overset{\eqref{eq4.1-28oct}}{\simeq_s} \sum_{i\in\mathbb N} \frac{\mu(B(y_i,t_i))}{t_i^s}\\
			&\overset{\eqref{lem4.11-3}}{\lesssim_s}\sum_{i\in \mathbb N}\int_{B(y_i,  t_i)}\int_{B(y_i, t_i)}\frac{|{\bf 1}_E(x)-{\bf 1}_E(y)|}{\mu(B_{x,y})d(x,y)^s}\,d\mu(x)\,d\mu(y)\\
			&\lesssim\|{\bf 1}_{E}\|_{\dot W^{s,1}(X)},
		\end{align*}
		since $\{B(y_i, r_i)\}_i$ are pairwise disjoint.
		This is the desired claim \eqref{lem4.11-2}. 
	\end{proof}
	
	\section{Proofs of main results}
	We recall  that $0<s<1$  and $(X,d,\mu)$ is assumed to be  a complete, connected metric measure space equipped with a doubling measure.
	\subsection{Proof of Theorem \ref{thm1}}\
	
	Theorem \ref{thm1} is obtained from Theorem \ref{thm1-20nov} below.
	In Theorem \ref{thm1-20nov}, we obtain a weak Harnack inequality both for 
	$s$-subminimizers and certain solutions of obstacle problems.
	The latter case will be applied when studying the
	weak Cartan property, see Lemma \ref{empty-intersection}.
	\begin{theorem}
		[Weak Harnack inequality]
		\label{thm1-20nov}

		Let $Q>s$ be a doubling dimension.
		Let $\Omega\subset X$ be a nonempty open set.
		Let  $x_0\in\Omega$, $0<r<R$ with $B(x_0,R)\subset\Omega$, and $k_0\in\mathbb R$.
		Assume that either
		\begin{enumerate}
			\item \label{Thm1.1-item-a}
			$u$ is an $s$-subminimizer 
			{in $\Omega$}; or
			\item \label{Thm1.1-item-b}
			$u$ is a $\mathcal{K}_{\psi,0}(\Omega)$-solution  where $\psi:\Omega\to[-\infty,\infty)$ is a function such that  $\psi\leq k_0$ a.e.\ on $B(x_0,R)$.
		\end{enumerate} 
		Then there exists a constant $C_3>0$  depending on $s$ and the doubling constant such that 
		\[
		\esssup_{y\in B(x_0,r)}(u(y)-k_0)\leq C_3 \left( \frac{R}{R-r}\right)^Q \dashint_{B(x_0,R)} (u-k_0)_+\,d\mu.
		\]
		
		Moreover, if we additionally assume that $\mu$ supports a $1$-Poincar\'e inequality, then the constant $C_3$ can be chosen as
		\[
		C_3 = \left(\frac{C_3'}{s}\right)^{Q/s},
		\]
		where $C_3' >0$ depends only on the doubling and Poincar\'e constants.
		
	\end{theorem}

	First we note the following standard connection between subminimizers and subsolutions.
	\begin{lemma}\label{lem1.3-1oct}
		If $f$ is an $s$-subminimizer 
		{in a given nonempty open set $\Omega$,}
		where $0<s<1$, then
		for every nonpositive  function  $\eta\in W^{s,1}_0(\Omega)$,
		{
			\begin{align*}
				&0\geq  \left( \int_\Omega\int_\Omega +2\int_\Omega\int_{\Omega^c} \right) \frac{ \frac{u(x)-u(y)}{|u(x)-u(y)|} (\eta(y)-\eta(x)) }{\mu(B_{x,y}) d(x,y)^s} 
				{\bf 1}_{\{ u(x)\neq u(y)\}} \,d\mu(x)\,d\mu(y)\\
				&\quad  - \left( \int_\Omega\int_\Omega +2\int_\Omega\int_{\Omega^c} \right) \frac{  |\eta(x)-\eta(y)| }{\mu(B_{x,y}) d(x,y)^s} {\bf 1}_{\{ u(x)= u(y)\}} \,d\mu(x)\,d\mu(y).
			\end{align*}
		}
		
	\end{lemma}
	\begin{proof}
		By Fatou's lemma, and since $F(u, u+t\eta)\leq 0$ for each $t>0$, we obtain that
		\begin{align*}
			\left( \int_\Omega\int_\Omega +2\int_\Omega\int_{\Omega^c} \right) &  \liminf_{t\to 0^+} \frac{{|u(x)-u(y)| - |(u+t\eta)(x)-(u+t\eta)(y)| }}{t} \frac{d\mu(x)d\mu(y)}{\mu(B_{x,y})d(x,y)^s} \\
			\leq & \liminf_{t\to 0^+} \frac{F(u,u+t\eta)}{t} \leq 0.
		\end{align*}
		This gives the desired claim because if $u(x)= u(y)$ then 
		\[
		\frac{{|u(x)-u(y)| - |(u+t\eta)(x)-(u+t\eta)(y)| }}{t} =-|\eta(x)-\eta(y)|,
		\]
		and if $u(x)\neq u(y)$ then
		
		\begin{align*}
			&  \liminf_{t\to 0^+}  \frac{{|u(x)-u(y)| - |(u+t\eta)(x)-(u+t\eta)(y)| }}{t} \\
			=&	-\frac{ u(x)-u(y)}{ |u(x)-u(y)| } (\eta(x)-\eta(y)).
		\end{align*}
	\end{proof}
	
	Toward proving the weak Harnack inequality, we first show the following
	Caccioppoli--De~Giorgi type inequality.
	
	\begin{proposition}\label{lem2.2-13oct}

		Let $\Omega$ be a nonempty open set.
		Let $k\in\mathbb{R}$, $x_0\in \Omega$, $0<\rho<R$ with $B(x_0,R)\subset \Omega$, and set $r:=\frac{R+\rho}{2}$.
		Assume that either
		\begin{enumerate}
			\item \label{item-a}
			$u$ is an $s$-subminimizer 
			{in $\Omega$; or }
			\item \label{item-b}
			$u$ is a $\mathcal{K}_{\psi,0}(\Omega)$-solution where $\psi:\Omega\to[-\infty,\infty)$ is a function such that $\psi\leq k$
			a.e.\ on $B(x_0,R)$.
		\end{enumerate}
		Let $\phi$ be a $\frac{1}{r-\rho}$-Lipschitz function with $0\leq \phi\leq 1$, $\phi\equiv 1 $ on $B(x_0,\rho)$ and $\phi\equiv  0$ on $ X\setminus B(x_0,r)$.
		Then we have the following Caccioppoli--De~Giorgi type inequality:  
		there exists a constant $C>0$, depending only on the doubling constant, such that
		\begin{equation}\label{eq3.2-15mar}
			\|(u-k)_+\phi\|_{\dot W^{s,1}(X)}
			\leq 
			\frac{C}{s(1-s)}
			\frac{1}{(R-\rho)^s}
			\int_{B(x_0,R)} (u-k)_+\, d\mu.
		\end{equation}
		
	\end{proposition}
	
	\begin{proof}
		Denote $B:=B(x_0,R)$. 
		Let $\eta:=-\phi (u-k)_+$.  We first prove that if  either  \ref{item-a} or \ref{item-b} holds, then 
		{
			\begin{align}\notag
				0\geq&\, \left( \int_\Omega\int_\Omega +2\int_\Omega\int_{\Omega^c} \right) \frac{\frac{u(x)-u(y)}{|u(x)-u(y)|} (\eta(y)-\eta(x)) }{\mu(B_{x,y}) d(x,y)^s} {\bf 1}_{\{u(x)\neq u(y) \}} \,d\mu(x)\,d\mu(y)\\
				&\, - \left( \int_\Omega\int_\Omega +2\int_\Omega\int_{\Omega^c} \right) \frac{ |\eta(x)-\eta(y)| }{\mu(B_{x,y}) d(x,y)^s} {\bf 1}_{\{u(x)=u(y)\}}\,d\mu(x)\,d\mu(y) \label{eq3.2-20nov}.
			\end{align}
		}
		By Lemma \ref{lem1.3-1oct}, if $u$ satisfies \ref{item-a}, then  \eqref{eq3.2-20nov}  holds.
		We now consider the case \ref{item-b}. Let $u$ be a $\mathcal{K}_{\psi,0}(\Omega)$-solution,
		and $\psi\leq k$ a.e. on $B(x_0,R)$. Hence $u\in \dot W^{s,1}(X)$,
		$u\geq \psi$ a.e. on $\Omega$, and $u=0$ a.e. on $X\setminus\Omega$.
		Let $0<t<\frac{1}{2}$. Set $v:=u+t\eta=u-t\phi(u-k)_+$. Then for a.e. $x\in\Omega$,
		\begin{align*}
			v(x)& =  \begin{cases}
				(1-t\phi(x))u(x)+t\phi(x) k& {\rm \ \ if\ \ }u(x)>k {\rm \ and\ }x\in B,\\
				u(x)& {\rm \ \ otherwise\ \ }
			\end{cases}\\
			&\geq   \psi(x)
		\end{align*}
		since $u\geq \psi$ a.e. on $\Omega$, $k\geq \psi$ a.e. on $B$, and $0\leq t\phi\leq 1$.
		Moreover, we have $v\in \dot W^{s,1}(X)$ because 
		\[
		\|v\|_{\dot W^{s,1}(X)}
		\leq 
		{ \|u\|_{\dot W^{s,1}(X)} }
		+\|t\phi(u-k)_+\|_{\dot W^{s,1}(X)}
		\leq 2 \|u\|_{\dot W^{s,1}(X)}<\infty.
		\]
		Notice that   $v=u=0$ a.e. on $X\setminus\Omega$. 
		Combining this with the two above properties, we obtain that
		$v\in \mathcal K_{\psi,0}(\Omega)$. Hence since $u$ is a $\mathcal{K}_{\psi,0}(\Omega)$-solution,
		we have $\|u\|_{\dot W^{s,1}(X)}\leq \|v\|_{\dot W^{s,1}(X)}=\|u+t\eta\|_{\dot W^{s,1}(X)}$.
		Since $v=u=0$ a.e. on $X\setminus\Omega$, it follows that 
		{ $0\geq F(u,u+t\eta)$ for  all $0<t<\frac{1}{2}$ because 
			\begin{align*}
				0
				&\geq  \int_X\int_X \frac{|u(x)-u(y)|-|v(x)-v(y)|}{\mu(B_{x,y})d(x,y)^s}\,d\mu(x)\,d\mu(y)\\
				&=\left(\int_{\Omega}\int_{\Omega}+2\int_{\Omega}\int_{\Omega^c} \right)\frac{|u(x)-u(y)|-|v(x)-v(y)|}{\mu(B_{x,y})d(x,y)^s}\,d\mu(x)\,d\mu(y)\\
				& = F(u,u+t\eta)
			\end{align*}
			for all $0<t<\frac{1}{2}$. 
		}
		Using this, the proof of Lemma \ref{lem1.3-1oct} gives the inequality \eqref{eq3.2-20nov}.

		Next, by the fact that $\eta\equiv 0$ on $B^c$, we have
		{
			\begin{align*}
				&- \left( \int_\Omega\int_\Omega +2\int_\Omega\int_{\Omega^c} \right) \frac{ |\eta(x)-\eta(y)| }{\mu(B_{x,y}) d(x,y)^s} {\bf 1}_{\{u(x)=u(y)\}}\,d\mu(x)\,d\mu(y)\\
				&=-  \left( \int_B\int_{B} +2\int_{B}\int_{B^c} \right) \frac{ |\eta(x)-\eta(y)| }{\mu(B_{x,y}) d(x,y)^s} {\bf 1}_{\{u(x)=u(y)\}}\,d\mu(x)\,d\mu(y)\\
				&= -2\int_B\int_{B}  \frac{ \eta(y)-\eta(x) }{\mu(B_{x,y}) d(x,y)^s} {\bf 1}_{\{u(x)=u(y)\}}{\bf 1}_{\{ \eta(x)\leq \eta(y)\}}\,d\mu(x)\,d\mu(y)\notag \\
				&\quad   -2\int_{B} \int_{B^c} \frac{ \eta(y) }{\mu(B_{x,y}) d(x,y)^s} {\bf 1}_{\{u(x)=u(y)\}}\,d\mu(x)\,d\mu(y)
			\end{align*}
			and
			\begin{align*}
				&\left( \int_\Omega\int_\Omega+2\int_{\Omega}\int_{\Omega^c}\right) \frac{\frac{u(x)-u(y)}{|u(x)-u(y)|} (\eta(y)-\eta(x)) }{\mu(B_{x,y}) d(x,y)^s} {\bf 1}_{\{ u(x)\neq u(y)\}}\,d\mu(x)\,d\mu(y)\notag \\
				&= \left( \int_B\int_B+2\int_{B}\int_{B^c}\right) \frac{\frac{u(x)-u(y)}{|u(x)-u(y)|} (\eta(y)-\eta(x)) }{\mu(B_{x,y}) d(x,y)^s} {\bf 1}_{\{ u(x)\neq u(y)\}}\,d\mu(x)\,d\mu(y)\notag \\
				&= \int_B\int_B\frac{\frac{u(x)-u(y)}{|u(x)-u(y)|} (\eta(y)-\eta(x)) }{\mu(B_{x,y}) d(x,y)^s} {\bf 1}_{\{ u(x)\neq u(y)\}}\,d\mu(x)\,d\mu(y)  \notag  \\
				&\quad +2\int_{B}\int_{B^c} \frac{\frac{u(x)-u(y)}{|u(x)-u(y)|} \eta(y) }{\mu(B_{x,y}) d(x,y)^s} {\bf 1}_{\{ u(x)\neq u(y)\}}\,d\mu(x)\,d\mu(y).
			\end{align*}
			By \eqref{eq3.2-20nov}, it follows that
			\begin{align}\label{eq1-2oct}
				0\geq T_{\rm L} + T_{\rm NL}
			\end{align}
			where 
			\begin{align*}
				T_{\rm L} &:=  \int_B\int_B\frac{\frac{u(x)-u(y)}{|u(x)-u(y)|} (\eta(y)-\eta(x)) }{\mu(B_{x,y}) d(x,y)^s} {\bf 1}_{\{ u(x)\neq u(y)\}}\,d\mu(x)\,d\mu(y) \\
				& \quad -2\int_B\int_{B}  \frac{ \eta(y)-\eta(x) }{\mu(B_{x,y}) d(x,y)^s} {\bf 1}_{\{u(x)=u(y)\}}{\bf 1}_{\{ \eta(x)\leq \eta(y)\}}\,d\mu(x)\,d\mu(y)
			\end{align*}
			and
			\begin{align*}
				T_{\rm NL} &:= 2\int_{B}\int_{B^c} \frac{\frac{u(x)-u(y)}{|u(x)-u(y)|} \eta(y) }{\mu(B_{x,y}) d(x,y)^s} {\bf 1}_{\{ u(x)\neq u(y)\}}\,d\mu(x)\,d\mu(y)\\
				& \quad -2\int_{B} \int_{B^c} \frac{ \eta(y) }{\mu(B_{x,y}) d(x,y)^s} {\bf 1}_{\{u(x)=u(y)\}}\,d\mu(x)\,d\mu(y).
			\end{align*}
		}
		We need to estimate the ``local'' term $T_{\mathrm{L}}$ and the ``nonlocal''
		term $T_{\mathrm{NL}}$.  
		{We  estimate the first term of $T_{\rm L}$, denoted $T_{\rm L,1}$, that} 
		\begin{align*}
			T_{\rm L, 1} & := \int_B\int_B\frac{\frac{u(x)-u(y)}{|u(x)-u(y)|} (\eta(y)-\eta(x)) }{\mu(B_{x,y}) d(x,y)^s} {\bf 1}_{\{ u(x)\neq  u(y)\}} \,d\mu(x)\,d\mu(y) \\
			&=\int_B\int_B\frac{\frac{u(x)-u(y)}{|u(x)-u(y)|} (\eta(y)-\eta(x)) }{\mu(B_{x,y}) d(x,y)^s} {\bf 1}_{\{ u(x)> u(y)\}} \,d\mu(x)\,d\mu(y)  \notag \\
			&\quad +\int_B\int_B\frac{\frac{u(x)-u(y)}{|u(x)-u(y)|} (\eta(y)-\eta(x)) }{\mu(B_{x,y}) d(x,y)^s} {\bf 1}_{\{ u(x)< u(y)\}} \,d\mu(x)\,d\mu(y) \notag  \\
			&=  \int_B\int_B\frac{\eta(y)-\eta(x) }{\mu(B_{x,y}) d(x,y)^s} {\bf 1}_{\{ u(x)> u(y)\}} \,d\mu(x)\,d\mu(y)  \notag \\
			&\quad +\int_B\int_B\frac{ \eta(x)-\eta(y) }{\mu(B_{x,y}) d(x,y)^s} {\bf 1}_{\{ u(x)< u(y)\}} \,d\mu(x)\,d\mu(y)\notag \\
			&= 2  \int_B\int_B \frac{\eta(x)-\eta(y) }{\mu(B_{x,y}) d(x,y)^s} {\bf 1}_{\{ u(x)< u(y)\}} \,d\mu(x)\,d\mu(y)\notag 
		\end{align*}
		by a change of variables. From this, we have that
		\begin{align}
			T_{\rm L,1}
			&= 2  \int_B\int_B \frac{\eta(x)-\eta(y) }{\mu(B_{x,y}) d(x,y)^s} {\bf 1}_{\{ u(x)< u(y)\}} {\bf 1}_{\{\eta(x)\geq \eta(y)\}} \,d\mu(x)\,d\mu(y)\notag \\
			&\quad + 2  \int_B\int_B \frac{\eta(x)-\eta(y) }{\mu(B_{x,y}) d(x,y)^s} {\bf 1}_{\{ u(x)< u(y)\}} {\bf 1}_{\{\eta(x)\leq \eta(y)\}} \,d\mu(x)\,d\mu(y)\notag \\
			&=  2\int_B\int_B \frac{\eta(x)-\eta(y) }{\mu(B_{x,y}) d(x,y)^s} {\bf 1}_{\{ \eta(x)\geq \eta(y)\}} \,d\mu(x)\,d\mu(y)\notag \\ 
			&\quad - 2  \int_B\int_B \frac{\eta(x)-\eta(y) }{\mu(B_{x,y}) d(x,y)^s} {\bf 1}_{\{ u(x)= u(y)\}} {\bf 1}_{\{\eta(x)\geq \eta(y)\}} \,d\mu(x)\,d\mu(y) \notag \\
			&\quad - 2  \int_B\int_B \frac{\eta(x)-\eta(y) }{\mu(B_{x,y}) d(x,y)^s} {\bf 1}_{\{ u(x)> u(y)\}} {\bf 1}_{\{\eta(x)\geq \eta(y)\}} \,d\mu(x)\,d\mu(y) \notag \\
			&\quad + 2  \int_B\int_B \frac{\eta(x)-\eta(y) }{\mu(B_{x,y}) d(x,y)^s} {\bf 1}_{\{ u(x)< u(y)\}} {\bf 1}_{\{\eta(x)\leq \eta(y)\}} \,d\mu(x)\,d\mu(y) \notag \\
			&=   \int_B\int_B \frac{|\eta(x)-\eta(y)| }{\mu(B_{x,y}) d(x,y)^s}  \,d\mu(x)\,d\mu(y)\notag \\ 
			&\quad - 2  \int_B\int_B \frac{\eta(x)-\eta(y) }{\mu(B_{x,y}) d(x,y)^s} {\bf 1}_{\{ u(x)= u(y)\}} {\bf 1}_{\{\eta(x)\geq \eta(y)\}} \,d\mu(x)\,d\mu(y) \notag \\
			&\quad + 4 \int_B\int_B \frac{\eta(x)-\eta(y) }{\mu(B_{x,y}) d(x,y)^s} {\bf 1}_{\{ u(x)< u(y)\}} {\bf 1}_{\{\eta(x)\leq \eta(y)\}} \,d\mu(x)\,d\mu(y) .\notag
		\end{align}
		{
			Hence
			\begin{align}
				\notag T_{\rm L}&= T_{\rm L,1}   -2\int_B\int_{B}  \frac{ \eta(y)-\eta(x) }{\mu(B_{x,y}) d(x,y)^s} {\bf 1}_{\{u(x)=u(y)\}}{\bf 1}_{\{ \eta(x)\leq \eta(y)\}}\,d\mu(x)\,d\mu(y)\\
				& \geq  \int_B\int_B \frac{|\eta(x)-\eta(y)| }{\mu(B_{x,y}) d(x,y)^s}  \,d\mu(x)\,d\mu(y)\notag\\
				\label{eq2-2oct} & \quad - 8 \int_B\int_B \frac{\eta(y)-\eta(x) }{\mu(B_{x,y}) d(x,y)^s} {\bf 1}_{\{ u(x)\leq  u(y)\}} {\bf 1}_{\{\eta(x)\leq \eta(y)\}} \,d\mu(x)\,d\mu(y).
			\end{align}
		}
		Next, for the non-local term $T_{\rm NL}$,
		we estimate 
		\begin{align*}
			T_{\rm NL}&\geq  -2 \int_{B} \eta(y) \left(\int_{B^c} \frac{1}{\mu(B_{x,y}) d(x,y)^s}\,d\mu(x) \right) d\mu(y)\\
			&=  -2 \int_{B} (u-k)_+(y)\phi(y) \left(\int_{B^c} \frac{1}{\mu(B_{x,y}) d(x,y)^s}\,d\mu(x) \right) d\mu(y).
		\end{align*}
		Setting $r_y=R-d(x_0,y)$ for each $y\in B$, we have that  $B(y,r_y)\subset B$
		for every $y\in B$. It follows
		from the doubling property that  for every $y\in B$,
		\begin{align}\label{eq2-3Lemma2}
			\int_{B^c} \frac{1}{\mu(B_{x,y}) d(x,y)^s}\,d\mu(x)
			&\leq \int_{X\setminus B(y,r_y)} \frac{d\mu(x)}{\mu(B_{x,y}) d(x,y)^s}\notag\\
			&=  \sum_{j=0}^\infty \int_{B(y,\,2^{j+1}r_y)\setminus B(y,\, 2^jr_y)} \frac{d\mu(x)}{\mu(B_{x,y}) d(x,y)^s}\notag\\
			&\simeq   \sum_{j=0}^\infty \frac{1}{(2^j r_y)^s}= \frac{1}{1-2^{-s}}\frac{1}{r_y^s}\notag\\
			& \lesssim \frac{1}{s}\frac{1}{(R-d(x_0,y))^s}.
		\end{align}
		From the last two estimates, we get
		\begin{align}
			T_{\rm NL}
			&\gtrsim  - \frac{1}{s}\int_{B} \frac{(u-k)_+(y)\phi(y)}{(R-d(x_0,y))^s}\,d\mu(y) \notag\\ 
			&=  - \frac{1}{s}\int_{B(x_0,r)} \frac{(u-k)_+(y)\phi(y)}{(R-d(x_0,y))^s}\,d\mu(y)
			\quad\textrm{since $\phi|_{X\setminus B(x_0,r)}\equiv 0$} \notag\\ 
			& \geq - \frac{2^s}{s}\int_{B(x_0,r)} \frac{(u-k)_+(y)\phi(y)}{(R-\rho)^s}\,d\mu(y), \notag
		\end{align}
		because for every $y\in B(x_0,r)$, we have from  $r=\frac{R+\rho}{2}$ that
		\[
		R-d(x_0,y)\geq R-r = \frac{R-\rho}{2}.
		\]
		Combining this with \eqref{eq1-2oct}-\eqref{eq2-2oct}, we get
		{
			\begin{align*}
				0\gtrsim &  \int_B\int_B \frac{|\eta(x)-\eta(y)| }{\mu(B_{x,y}) d(x,y)^s}  \,d\mu(x)\,d\mu(y)\\
				&-   \int_B\int_B \frac{\eta(y)-\eta(x) }{\mu(B_{x,y}) d(x,y)^s} {\bf 1}_{\{ u(x)\leq u(y)\}} {\bf 1}_{\{\eta(x)\leq \eta(y)\}} \,d\mu(x)\,d\mu(y)\\
				& - \frac{1}{s} \int_{B} \frac{(u-k)_+(x)\phi(x)}{(R-\rho)^s} \,d\mu(x).
			\end{align*}
		}
		We have
		\begin{align*}
			&  \int_B\int_B \frac{\eta(y)-\eta(x) } {\mu(B_{x,y}) d(x,y)^s} {\bf 1}_{\{ u(x)\leq u(y)\}} {\bf 1}_{\{\eta(x)\leq \eta(y)\}} \,d\mu(x)\,d\mu(y)\\
			&=  \int_B\int_B \frac{(u-k)_+(y) (\phi(x)-\phi(y))}{\mu(B_{x,y}) d(x,y)^s} {\bf 1}_{\{ u(x)\leq u(y)\}} {\bf 1}_{\{\eta(x)\leq \eta(y)\}} \,d\mu(x)\,d\mu(y)\\
			&\quad +  \int_B\int_B \frac{((u-k)_+(x) -(u-k)_+(y))\phi(x)}{\mu(B_{x,y}) d(x,y)^s} {\bf 1}_{\{ u(x)\leq u(y)\}} {\bf 1}_{\{\eta(x)\leq \eta(y)\}} \,d\mu(x)\,d\mu(y)\\
			&\leq   \int_B\int_B \frac{\max\{(u-k)_+(x), (u-k)_+(y) \} |\phi(y)-\phi(x)|}{\mu(B_{x,y}) d(x,y)^s}  \,d\mu(x)\,d\mu(y),
		\end{align*}
		where the last line is obtained because the second term in the third line is negative. Combining the last two estimates, we conclude that		
		{
			\begin{align}
				0\gtrsim&  \int_B\int_B \frac{|\eta(x)-\eta(y)| }{\mu(B_{x,y}) d(x,y)^s}  \,d\mu(x)\,d\mu(y) \notag \\ \notag
				&-   \int_B\int_B \frac{\max\{(u-k)_+(x), (u-k)_+(y) \} |\phi(y)-\phi(x)|}{\mu(B_{x,y}) d(x,y)^s}  \,d\mu(x)\,d\mu(y)
				\\ & - \frac{1}{s} \int_{B} \frac{ (u-k)_+(x)\phi(x) }{(R-\rho)^s}\,d\mu(x).\label{eq2.3-13oct}
			\end{align}
		}
		Notice that $\phi$ is $\frac{1}{r-\rho}$-Lipschitz, and so $\frac{2}{R-\rho}$-Lipschitz because $r=\frac{R+\rho}{2}$. We have from \eqref{eq2.6-29oct} that
		\[
		\int_{B}  \frac{  |\phi(y)-\phi(x)|}{\mu(B_{x,y}) d(x,y)^s}  \,d\mu(x)
		\lesssim \frac{1}{s(1-s)}\frac{1}{(R-\rho)^s}.
		\]
		Inserting this into \eqref{eq2.3-13oct}, we obtain 
		{
			\[
			0\gtrsim   \int_B\int_B \frac{|\eta(x)-\eta(y)| }{\mu(B_{x,y}) d(x,y)^s}  \,d\mu(x)\,d\mu(y) - \frac{1}{s(1-s)}\frac{1}{(R-\rho)^s}\int_{B} (u-k)_+\, d\mu.
			\]
			Notice from the definition of subminimizers and $\mathcal K_{\psi,0}$-solution that $u\in \dot W^{s,1}(\Omega)$, and so $u$ is locally integrable on $\Omega$ by \eqref{eq2.1-11dec}. In particular, $\int_{B} (u-k)_+\, d\mu$ is finite, and hence by adding both sides of the above estimate  by the finite quantity $ \frac{1}{s(1-s)}
			\frac{1}{(R-\rho)^s}
			\int_{B} (u-k)_+\, d\mu$, 
			we obtain that
			\begin{align}
				\|(u-k)_+\phi\|_{\dot W^{s,1}(B(x_0,R))}
				\lesssim
				\frac{1}{s(1-s)}
				\frac{1}{(R-\rho)^s}
				\int_{B(x_0,R)} (u-k)_+\, d\mu.
				\label{eq2.1-19oct}
			\end{align}
			By the estimate  \eqref{eq2-3Lemma2}, it follows from $\phi\equiv 0$ on $X\setminus B(x_0,R)$ that
			\begin{align*}
				&\quad \|(u-k)_+\phi\|_{\dot W^{s,1}(X)}\\
				&= \|(u-k)_+\phi\|_{\dot W^{s,1}(B(x_0,R))} \\
				&\quad + 2 \int_{B(x_0,R)} \left( \int_{X\setminus B(x_0,R)} \frac{(u-k)_+(y)\phi(y)}{\mu(B_{x,y})d(x,y)^s}d\mu(x) \right)d\mu(y)\\
				&\overset{\eqref{eq2.1-19oct} \&  \eqref{eq2-3Lemma2}}{\lesssim} \frac{1}{s(1-s)}\frac{1}{(R-\rho)^s}\int_{B(x_0,R)} (u-k)_+\, d\mu\\
				&\quad + \frac{1}{s(1-s)}\int_{B(x_0,R)}  \frac{(u-k)_+(y)\phi(y)}{{( R- d(y,x_0))^s}}\, d\mu(y)\\
				& {\lesssim} \, \frac{1}{s(1-s)}\frac{1}{(R-\rho)^s}\int_{B(x_0,R)} (u-k)_+\, d\mu
			\end{align*}
			where the last estimate is obtained since $0\leq \phi\leq 1,$  $\phi\equiv  0$ on $ X\setminus B(x_0,r)$, and  $R-d(y,x_0)\geq R-r=\frac{R-\rho}{2}$ for $y\in B(x_0,r)$. The desired claim \eqref{eq3.2-15mar}  follows.
		}
	\end{proof}
	
	\begin{proof}[Proof of Theorem \ref{thm1-20nov}]
		We first consider the case where $\mu$ supports a $1$-Poincar\'e inequality.
		
		Let $r\leq \rho<r_1<r_2\leq R$ such that $r_2\leq 2\rho$ and $r_1=\frac{r_2+\rho}{2}$.
		Let $\phi$ be $\frac{1}{r_1-\rho}$-Lipschitz such that $0\leq \phi\leq 1, \phi\equiv 1$
		on $B({x_0,\rho}),$ $\phi\equiv 0$ on $B(x_0,r_1)^c$. Notice that $\phi$ is
		$\frac{2}{r_2-\rho}$-Lipschitz because $r_1=\frac{r_2+\rho}{2}$.
		Let further $k>l\geq k_0$. Applying \eqref{Poincare-local}
		to $(u-k)_+\phi$, and by Proposition \ref{lem2.2-13oct}, it follows that 
		\begin{align*}
			\left( \dashint_{B(x_0,r_2)} |(u-k)_+ \phi|^{\frac{Q}{Q-s}} \, d\mu \right)^{\frac{Q-s}{Q}} 
			&\overset{\eqref{Poincare-local}}{\lesssim}  \frac{(1-s) r_2^s}{\mu(B(x_0,r_2))}  \|  (u-k)_+ \phi\|_{\dot W^{s,1}(X)}  \\
			&\overset{\eqref{eq3.2-15mar}}{ \lesssim }  \frac{r_2^s}{s(r_2-\rho)^s} \dashint_{B(x_0,r_2)} (u-k)_+\,  d\mu.
		\end{align*}
		This implies that
		\begin{align}
			&\quad \left( \dashint_{B(x_0,r_2)} |(u-k)_+ \phi|^{\frac{Q}{Q-s}} \, d\mu \right)^{\frac{Q-s}{Q}}\notag \\
			&\lesssim  \frac{r_2^s}{s(r_2-\rho)^s} \dashint_{B(x_0,r_2)} (u-k)_+\,  d\mu\notag \\
			&\leq  \frac{r_2^s}{s(r_2-\rho)^s} \dashint_{B(x_0,r_2)} (u-l)_+\,  d\mu\quad 
			\textrm{since\ }k>l.  \label{eq3-3oct}
		\end{align}
		Notice from  $r_2\leq 2\rho$ that  $\mu(B(x_0,\rho))$, $\mu(B(x_0,r_2))$ are comparable.
		We have from the H\"older inequality that
		\begin{align*}
			&\quad \dashint_{B(x_0,\rho)} (u-k)_+\,d\mu\\
			&\lesssim   \dashint_{B(x_0,r_2)} (u-k)_+\phi \, d\mu, \quad \text{since $\phi|_{B(x_0,\rho)}\equiv 1$,}\\
			&\leq \left( \dashint_{B(x_0,r_2)} |(u-k)_+\phi|^{\frac{Q}{Q-s}} \,d\mu\right)^{\frac{Q-s}{Q}} \left( \dashint_{B(x_0,r_2)} {\bf 1}_{\{ u\geq k\}}\, d\mu\right)^{\frac{s}{Q}}\\
			&\leq \left( \dashint_{B(x_0,r_2)} |(u-k)_+\phi|^{\frac{Q}{Q-s}} \,d\mu\right)^{\frac{Q-s}{Q}} \left( \dashint_{B(x_0,r_2)} \frac{(u-l)_+}{k-l}\, d\mu\right)^{\frac{s}{Q}}\\
			&= \frac{1}{(k-l)^{\frac{s}{Q}}}\left( \dashint_{B(x_0,r_2)} |(u-k)_+\phi|^{\frac{Q}{Q-s}} \,d\mu\right)^{\frac{Q-s}{Q}} \left( \dashint_{B(x_0,r_2)} {(u-l)_+}\, d\mu\right)^{\frac{s}{Q}}\\
			&\overset{\eqref{eq3-3oct}}{\lesssim}
			\frac{1}{s}
			\frac{1}{(k-l)^{\frac{s}{Q}}} \frac{r_2^s}{(r_{2}-\rho)^s} \left( \dashint_{B(x_0,r_2)} {(u-l)_+}\, d\mu\right)^{1+\frac{s}{Q}}.
		\end{align*}
		Hence there is a constant $C>0$  depending only  on the doubling and Poincar\'e constants such that 
		\begin{equation}\label{eq4-3oct}
			u(k,\rho)\leq \frac{C}{s}\frac{r_2^s}{(r_2-\rho)^s}\frac{1}{(k-l)^{\frac{s}{Q}}} u(l,r_2)^{1+\frac{s}{Q}}
		\end{equation}
		where
		\[
		u(k,\rho):=\dashint_{B(x_0,\rho)}(u-k)_+\,d\mu.
		\]
		
		For $n=0,1,\ldots,$ let $\rho_n=r+2^{-n}(R-r)\leq R$ and $k_n=k_0+d(1-2^{-n})\geq k_0$ where $d>0$ is defined in \eqref{eq2.8-13oct}. Then
		\[
		\rho_0=R,\, \rho_n\to r,\, k_n\to k_0+d {\rm \ as\ }n\to\infty.
		\]

		We now show by induction that 
		\begin{equation}\label{eq1.5-3oct}
			u(k_n,\rho_n)\leq (2^{-n})^{1+Q} u(k_0,R).
		\end{equation}
		If $n=0$, then \eqref{eq1.5-3oct} is trivial.
		Assume that \eqref{eq1.5-3oct} holds for the index $n$.
		Applying \eqref{eq4-3oct} 
		{to}
		$r_2=\rho_n, \rho=\rho_{n+1}, k=k_{n+1}, l=k_n$,
		we obtain from $k_{n+1}-k_n=2^{-n-1}d$ and $\rho_n-\rho_{n+1}=2^{-n-1}(R-r)$ that
		\begin{align*}
			u(k_{n+1},{\rho_{n+1}})
			&\overset{\eqref{eq4-3oct}}{\leq}  \frac{ C\cdot R^s}{s(\rho_n-\rho_{n+1})^s (k_{n+1}-k_n)^{\frac{s}{Q}}} u(k_{n}, \rho_n)^{1+\frac{s}{Q}}\\
			&\overset{\eqref{eq1.5-3oct}}{\leq } \frac{C\cdot R^s}{s(2^{-n-1} (R-r))^s (2^{-n-1}d)^{\frac{s}{Q}}} (2^{-n})^{(1+Q)\frac{Q+s}{Q}} u(k_0,R)^{\frac{Q+s}{Q}}\\
			&= (2^{-(n+1)})^{1+Q} u(k_0,R),
		\end{align*}
		where 
		\begin{align}
			d:= & \left( \frac{C}{s} \left(\frac{R}{R-r} \right)^s2^{1+s+Q+\frac{s}{Q}} \right)^{\frac{Q}{s}} u(k_0,R). \label{eq2.8-13oct}
		\end{align}
		So \eqref{eq1.5-3oct} is proved with $d$ as above.
		Letting $n\to\infty$, we conclude that $u(k_0+d,r)=0$ and so, for a.e. $x\in B(x_0,r)$,
		\begin{align}
			u(x)-k_0\leq \, d \overset{\eqref{eq2.8-13oct}}{\simeq} \left(\frac{C}{s} \right)^{\frac{Q}{s}} \left(\frac{R}{R-r}\right)^{Q}  \dashint_{B(x_0,R)}(u-k_0)_+\,d\mu \label{eq3.13-10dec},
		\end{align}
		which is the proof for the case where $\mu$ supports a $1$-Poincar\'e inequality.

		In the case where $(X,d,\mu)$ is a complete, connected metric measure space equipped with a doubling measure, the conclusion follows by applying \eqref{Poincare-local-2} with constants depending on $s$.
	\end{proof}
	
	\subsection{Proof of Corollary \ref{cor3.3-16oct}}\
	
	\begin{proof}[Proof of Corollary \ref{cor3.3-16oct}]
		We prove only the lower semicontinuity of \(u^\wedge\);
		the proof of the upper semicontinuity of \(u^\vee\) is similar.
		
		Let $u\in \dot W^{s,1}(\Omega)$ be an $s$-superminimizer in $\Omega$.
		Let $x\in \Omega$ be arbitrary. Notice from \eqref{eq2.1-11dec} that
		$u$ is locally integrable on $\Omega$. Hence the weak Harnack inequality in
		Theorem \ref{thm1} (applied with $k_0=0$ and the $s$-subminimizer $-u$ in $\Omega$) gives that
		$u$ is essentially bounded from below in a neighbourhood of $x$, i.e. there are
		$\beta\in\mathbb R$ and a sufficiently small $R>0$ with $B(x,2R)\Subset\Omega$ such that
		$u\geq \beta$ a.e. in $B(x,R)$. Then $u^\wedge(x)>-\infty$.
		Let $-\infty<t\leq u^\wedge(x)$ be a real number and $\varepsilon>0$.
		By Theorem \ref{thm1} applied with $k_0=0$ and the $s$-subminimizer $(t-u)$ on $\Omega$, we obtain
		that there is a constant $C>0$ depending only on $s$ and the doubling 
		such that for every $0<r<R$,
		\begin{align*}
			&\quad \esssup_{B(x,r/2)}(t-u)\\
			&\leq  C\dashint_{B(x,r)} (t-u)_+\,d\mu\\
			&= \frac{C}{\mu(B(x,r))} \left( \int_{B(x,r)\cap \{ t-\varepsilon
				\leq u<t\}}(t-u)\,d\mu+ \int_{B(x,r)\cap \{ u<t-\varepsilon\}}(t-u)\,d\mu \right)\\
			&\leq  C\cdot \varepsilon + C(t-\beta)_+ \frac{\mu(\{u<t-\varepsilon\}\cap B(x,r)) }{\mu(B(x,r))},
		\end{align*}
		because $u\geq \beta$ a.e. in $B(x,R)$. Since the second term of the right-hand side of the above
		estimate converges to $0$ as $r\to 0^+$ by the definition of $u^\wedge(x)$,
		letting $r\to 0^+$ we obtain that 
		\[
		\lim_{r\to0^+}\esssup_{B(x,r/2)}(t-u) \leq C\cdot \varepsilon.
		\]
		Then there is $r_{t,\varepsilon}>0$ sufficiently small
		that $u(y)\geq t-2C\cdot \varepsilon$  for a.e.
		$y\in B(x,r_{t,\varepsilon}/2)$. It follows that $u^\wedge(y)\geq t-2C\cdot\varepsilon$ for each
		$y\in B(x,r_{t,\varepsilon}/2)$. If $u^\wedge(x)\in\mathbb R$, we pick $t=u^\wedge(x)$ and
		hence $u^\wedge$ is lower semicontinuous at $x$. If $u^\wedge(x)=\infty$, by letting $t\to\infty$,
		we get $u^\wedge(y)\to\infty$ as $y\to x$, and so $u^\wedge$ is again
		lower semicontinuous at $x$.
		Since $x$ is arbitrary in $\Omega$, we obtain the  desired claim.
	\end{proof}
	
	\subsection{Proof of Theorem \ref{weakCartan-main}}\

	We denote the interaction operator between two disjoint measurable sets $E,F\subset X$ by
	\[
	L_s(E,F):=\int_E\int_F  \frac{1}{\mu(B_{x,y})d(x,y)^s}\,d\mu(x)\,d\mu(y).
	\]
	We denote the \emph{$s$-perimeter} of a measurable set $E\subset X$ in an open set $\Omega$ by $$P_{s}(E,\Omega):=L_{s}(E\cap\Omega, \Omega\setminus E)=\frac{1}{2}\|{\bf  1}_E \|_{\dot W^{s,1}(\Omega)}.$$
	When $\Omega=X$, we simply write $P_s(E):=P_s(E,X)$.
	Observe that for two disjoint sets $F_1,F_2\subset X$ with $P_s(F_1)<\infty$ and $P_s(F_2)<\infty$, we have
	\begin{align} 
		\notag		P_s(F_1\cup F_2) & = \int_{F_1\cup F_2}\int_{X\setminus (F_1\cup F_2)}  \frac{1}{\mu(B_{x,y})d(x,y)^s}\,d\mu(x)\,d\mu(y)\\
		\notag & =  \left(\int_{F_1}\int_{X\setminus F_1} +  \int_{F_2}\int_{X\setminus F_2} - 2 \int_{F_1}\int_{F_2} \right) \frac{d\mu(x)\,d\mu(y)}{\mu(B_{x,y})d(x,y)^s}\, \\
		& =P_s(F_1)+P_s(F_2)  -2L_s(F_1,F_2)
		\label{eq:F1 and F2}
	\end{align}
	where the second line is obtained because $F_1, F_2$ are disjoint.
	Moreover, for two measurable sets $F_1, F_2\subset X$ with $P_s(F_1)<\infty$ and $P_s(F_2)<\infty$, we have
	\begin{equation}
		\label{eq3.25-dec}P_s(F_1\setminus F_2)\leq P_s(F_1)+P_s(F_2) <\infty.
	\end{equation}
	
	\begin{lemma}\label{nonlocal}

		Let 
		$x_0\in X$ and $r>0$.
		Then for each $0<c_0<1$, for sufficiently large $k\in\mathbb N$
		depending only on $c_0$, $s$ and the doubling constant we have for every measurable
		$E\subset B(x_0, 2^{-(k-3)}r)$ that
		\begin{align}\label{nonlocalestimate}
			L_{s}(E, X\setminus B(x_0,r))\leq c_0 L_s(E, X\setminus B(x_0,2^{-(k-3)}r)).
		\end{align}
	\end{lemma}
	\begin{proof}
		Since  $r-d(x_0,y)\geq r/2$ for $y\in B(x_0,r/2)$, for such $y$ we have that
		\begin{align*}
			\int_{X\setminus B(x_0,r)} \frac{1}{\mu(B_{x,y})d(x,y)^s}\,d\mu(x)
			\overset{\eqref{eq2-3Lemma2}}{\lesssim}\,  \frac{1}{s} \frac{1}{(r-d(x_0,y))^s}  
			\leq \,  \frac{2^s}{s\cdot r^s}.
		\end{align*}
		We denote $A_j(y):=B(y,2^{-j+1}r)\setminus B(y,2^{-j} r)$ for $j\in \mathbb N$. Since
		\[
		\lim_{j\to \infty}\int_{A_j(y)} \frac{1}{\mu(B_{x,y})d(x,y)^s}\,d\mu(x)=\infty,
		\]
		it follows from the above estimates that there is  $j_0\in\mathbb N$ depending only  on $s,c_0$ and the doubling constant such that for every $y\in B(x_0,r/2)$,
		\[
		\int_{X\setminus B(x_0,r)} \frac{1}{\mu(B_{x,y})d(x,y)^s}\,d\mu(x)
		\le c_0\int_{A_{j_0}(y)}\frac{1}{\mu(B_{x,y})d(x,y)^s}\,d\mu(x).
		\]
		Let $k:=100\cdot j_0$. Let $y\in B(x_0, 2^{-(k-3)}r)$ be arbitrary.
		Then the condition $k=100\cdot j_0$ and the definition of $A_{j_0}(y)$ give that for each  $y\in B(x_0,2^{-(k-3)}r)$,
		\[
		A_{j_0}(y)\subset X\setminus B(x_0, 2^{-(k-3)}r).
		\]
		Combining this with the above inequality, we obtain that 
		\[
		\int_{X\setminus B(x_0,r)} \frac{1} {\mu(B_{x,y})d(x,y)^s}\,d\mu(x)
		\leq c_0 \int_{X\setminus B(x_0, 2^{-(k-3)}r)} \frac{1}{\mu(B_{x,y})d(x,y)^s}\,d\mu(x).
		\]
		In particular, this holds for every $y\in E$ since $E\subset B(x_0, 2^{-(k-3)}r)$, which implies  \eqref{nonlocalestimate} by integrating both sides of the above inequality over $y\in E$.
	\end{proof}

	\begin{lemma}\label{lem3.15-dec}
		For almost every $r>0$, for $z\in X$,
		$
		{P_s(B(z,r))<\infty} .
		$
	\end{lemma}
	\begin{proof}
		Let $R>0$ be arbitrary. Let $f(x)$ be a $1$-Lipschitz function defined by $f(x):=d(z,x)$ if $d(z,x)\leq R$, and $f(x)=R$ if $d(z,r)\geq R$.
		First, we have from $f\equiv R$ on $X\setminus B(z,R)$ that
		\begin{align*}
			\|f\|_{\dot W^{s,1}(X)}&\leq 2 \int_{B(z,R)} \int_{X} \frac{|f(x)-f(y)|}{\mu(B_{x,y})d(x,y)^s}d\mu(x)d\mu(y)\\
			&\overset{\eqref{eq2.6-29oct}}\lesssim \int_{B(z,R)} \frac{1}{s(1-s)}  d\mu(x) <\infty
		\end{align*}
		because $f$ is $1$-Lipschitz and bounded.  By coarea formula \eqref{eq2.7-27nov}, this gives that 
		\begin{align*}
			\int_{-\infty}^\infty  \| {\bf 1}_{\{f<t\}} \|_{\dot W^{s,1}(X)} dt
			&=\int_{-\infty}^\infty \| {\bf 1}_{\{-f>-t\}} \|_{\dot W^{s,1}(X)} dt\\
			&\overset{\eqref{eq2.7-27nov}}{=}\|-f\|_{\dot W^{s,1}(X)}=\|f\|_{\dot W^{s,1}(X)}<\infty.
		\end{align*}
		Notice that $B(z,r)=\{ f<r\}$ for $0<r\leq R$. Hence it follows from Fubini's theorem that for almost every $R\geq r>0$, $${P_s(B(z,r))=\frac{1}{2} \| {\bf 1}_{B(z,r)}\|_{\dot W^{s,1}(X)}<\infty}.$$
		Since $R>0$ is arbitrary, this gives the desired claim.
	\end{proof}
	
	\begin{proposition}\label{empty-intersection}

		Let 
		$3<k\in\mathbb N$ be the constant from Lemma \ref{nonlocal} applied with $c_0=\frac{1}{4}$.
		Let $B':=B(x,R)$ be an open ball. Suppose that each $B_j:=B(x,2^{-j}R)$, where $j\in\mathbb N\cup\{0\}$, satisfies $P_s(B_j)<\infty$.
		Let $W'\subset X$ be such that 
		\begin{equation*}
			\sup_{0<t\le R} \frac{\caps(W'\cap B(x,t),B(x,2t))}{\caps(B(x,t),B(x,2t))}\le \frac{1}{2C_4},
		\end{equation*}
		where $C_4:=2^{k} C_1C_2C_3C_{\mu}^{k+1}$,
		and
		\begin{equation}\label{eq3.14-23Jan}
			W'\cap \left(\frac{9}{20}B_{lk}\setminus \frac{1}{2^k}B_{lk}\right)=\emptyset 
			\quad \text{for every $l\in \mathbb N\cup \{0\}$}.
		\end{equation}
		Here $C_1,C_2,C_3$ are
		the constants from \eqref{isoperimetric-set}, \eqref{eq4.1-28oct}, and
		Theorem \ref{thm1-20nov} respectively.
		
		Then for every  $\mathcal{K}_{W'\cap B',0}(2^{k-1}B')$-solution $E$, after modifying $E$
		in a set of zero measure, we have
		\begin{equation}\label{eq-empty-intersection}
			E\cap \big(\frac{1}{4}B_{lk}\setminus \frac{1}{2^{k-1}}B_{lk}\big)=\emptyset\quad \text{for every $l\in \mathbb N\cup\{0\}$.}
		\end{equation}
		Furthermore, for each $l\in\mathbb N\cup\{0\}$, 
		\begin{equation}
			P_s(E\cap B_{lk+k-1}) \le
			\caps(W'\cap B_{lk}, 2B_{lk}).
			\label{new-2feb}
		\end{equation}
	\end{proposition}
	\begin{proof}

		We prove this lemma by induction in  $l\in\mathbb N\cup\{0\}$. Set $F_{lk}:=\frac{1}{4}B_{lk}\setminus\frac{1}{2^{k-1}}B_{lk}$.
		For $l=0$, since $E$ is a $\mathcal{K}_{W'\cap B',0}(2^{k-1}B')$-solution, we have ${\bf 1}_E=0$ a.e. in $X\setminus (2^{k-1}B')$. 
		By the isoperimetric inequality \eqref{isoperimetric-set},
		we have
		\begin{align}
			\notag	\mu(E)
			&\overset{\eqref{isoperimetric-set}}{\le}C_1 (2^{k-1}R)^s \|{\bf 1}_E\|_{\dot W^{s,1}(X)}\\
			&\overset{\eqref{eq2.7-19nov}}{\le} 2^{(k-1)s}C_1 R^s \caps(W'\cap B',2^{k-1}B') \notag \\
			&\overset{\eqref{e4.2-29oct}}{\leq}2^{(k-1)s}C_1 R^s \caps(W'\cap B',2B').
			\label{eq3.16-23Jan}
		\end{align}
		Let 
		{$y\in F_0:=\frac{1}{4}B_0\setminus \frac{1}{2^{k-1}}B_0$}
		be arbitrary. Set $r:={R}/{2^k}$. Then
		\begin{equation}\label{eq3.17-23Jan}B(y,r)\subset \frac{9}{20}B'\setminus \frac{1}{2^k}B'.
		\end{equation}
		Since 
		{$0\geq {\bf 1}_{W'\cap B'}$}
		on $B(y,r)$ by \eqref{eq3.14-23Jan} and \eqref{eq3.17-23Jan}, applying Theorem \ref{thm1-20nov}\ref{Thm1.1-item-b} with $k_0=0$, we obtain that
		\begin{align*}
			\esssup_{B(y,r/2)}{\bf 1}_E
			&\le C_3  \dashint_{B(y,r)}{\bf 1}_E\,d\mu\\
			&\le   \frac{C_3\mu(E)}{\mu(B(y,r))}\\
			&\le C_3C_{\mu}^{k+1} \frac{\mu(E)}{\mu(B')}, 
			\quad \text{ since $B(y,r)\subset B(x,R)\subset 2^{k+1}B(y,r)$, }
			\\
			&\overset{\eqref{eq3.16-23Jan}}{\le}2^{(k-1)s} C_1C_3C_{\mu}^{k+1} \frac{R^s}{\mu(B')}\caps(W'\cap B',2B')\\
			&\overset{\eqref{eq4.1-28oct}}{\le }2^{(k-1)s} C_1C_2C_3C_{\mu}^{k+1} 
			\frac{\caps(W'\cap B',2B')}{\caps(B',2B')}.
		\end{align*}
		For every $y\in F_0$, it follows that
		\begin{equation*}
			{\bf 1}_{E}^\vee(y)
			\le C_4
			\frac{\caps(W'\cap B',2B')}{\caps(B',2B')}\le\frac{1}{2},
		\end{equation*}
		which implies by Lebesgue's differentiation theorem that $E\cap F_0=\emptyset$
		after modifying $E$ in a set of zero measure.
		
		Suppose that the claim \eqref{eq-empty-intersection} holds for
		$l\in\mathbb N\cup\{0\}$. It suffices to prove that it holds for $l+1$.
		
		We have $E\cap (\frac{1}{4}B_{lk}\setminus \frac{1}{2^{k-1}}B_{lk})=\emptyset$. 
		Let $E_{lk}$ be a $\mathcal{K}_{W'\cap B_{lk}}(2^{k-1}B_{lk}) $-solution.
		Note that the set $E_{lk}\cup (E\setminus B_{lk+k-1})$ is admissible for $\mathcal{K}_{W'\cap B',0}(2^{k-1}B')$-obstacle problem. Then
		\begin{align*}
			&P_s(E\cap B_{lk+k-1})+P_s(E\setminus B_{lk+k-1})-2L_s(E\cap B_{lk+k-1},E\setminus B_{lk+k-1})\\
			&\overset{\eqref{eq:F1 and F2}}{=} P_s(E)\\
			&\le P_s(E_{lk}\cup (E\setminus B_{lk+k-1}))\\
			&\overset{\eqref{eq:F1 and F2}}{\le } P_s(E_{lk})+P_s(E\setminus B_{lk+k-1}).
		\end{align*} 
		Notice from \eqref{eq4.1-28oct} and Lemma \ref{lemma4.6}
		that $P_s(E)<\infty$.
		By $\eqref{eq3.25-dec}$ and since $P_s(B_j)<\infty$ for each $j\in\mathbb N\cup \{ 0\}$ by the assumption, we have that $P_s(E\setminus B_{lk+k-1})\le P_s(E)+P_s(B_{lk+k-1})<\infty$, and thus we can subtract above to get
		\begin{equation}\label{weak-cartan-1}
			P_s(E\cap B_{lk+k-1})\le P_s(E_{lk})+2L_s(E\cap B_{lk+k-1}, E\setminus B_{lk+k-1}).
		\end{equation}
		On the other hand, by Lemma \ref{nonlocal} applied with $c_0=\frac{1}{4}$, we have that
		\begin{align*}
			&\quad L_s(E\cap B_{lk+k-1}, E\setminus B_{lk+k-1})\\
			&=L_s(E\cap B_{lk+k-1}, E\setminus B_{lk+2}),\quad \text{since }E\cap F_{lk}=\emptyset,\\
			&\le L_s(E\cap B_{lk+k-1}, X\setminus B_{lk+2})\\
			& \overset{\eqref{nonlocalestimate}}{\le }
			\frac{1}{4} L_s(E\cap B_{lk+k-1}, X\setminus B_{lk+k-1})\\
			&\le \frac{1}{4}P_s(E\cap B_{lk+k-1}).
		\end{align*}
		Combining this with $\eqref{weak-cartan-1}$ and recalling that $P_s(E_{lk})=\frac{1}{2}\|{\bf1}_{E_{lk}} \|_{\dot W^{s,1}(X)}$, we have that 
		\begin{align}
			P_s(E\cap B_{lk+k-1})
			&\le 2P_s(E_{lk})\notag\\
			&\le \caps(W'\cap B_{lk}, 2^{k-1}B_{lk}),\quad \text{by Lemma \ref{lemma4.6}},\notag\\
			&\overset{\eqref{e4.2-29oct}}{\le}\caps(W'\cap B_{lk}, 2B_{lk})
			\label{eq3.24-24Jan},
		\end{align}
		which implies that  the last claim \eqref{new-2feb} holds.
		Let $y\in F_{(l+1)k}=\frac{1}{4}B_{(l+1)k}\setminus \frac{1}{2^{k-1}}B_{(l+1)k}$ be arbitrary.
		Set $r:=2^{-(l+2)k}R $. Then 
		\begin{equation}\label{eq3.25-24Jan}
			B(y,r)\subset \frac{9}{20} B_{(l+1)k}\setminus \frac{1}{2^k}B_{(l+1)k}.
		\end{equation}
		By the isoperimetric inequality \eqref{isoperimetric-set} and the fact that $E\cap F_{lk}=\emptyset$, we obtain that
		\begin{align}
			\mu(E\cap \tfrac{1}{4}B_{lk})
			&=\mu(E\cap B_{lk+k-1})\notag\\
			&\le C_1\left(2^{-(lk+k-1)}R\right)^s \|{\bf 1}_{E\cap B_{lk+k-1}}\|_{\dot W^{s,1}(X)}\notag\\
			&\overset{\eqref{eq3.24-24Jan}}{\le }2C_1\left(2^{-(lk+k-1)}R\right)^s\caps(W'\cap B_{lk}, 2B_{lk}).
			\label{eq3.26-24Jan}
		\end{align}
		Since 
		{$0\geq {\bf 1}_{W'\cap B'}$}
		on $B(y,r)$ by \eqref{eq3.14-23Jan} and \eqref{eq3.25-24Jan}, applying Theorem \ref{thm1-20nov}\ref{Thm1.1-item-b} with $k_0=0$, we obtain that
		\begin{align*}
			\esssup_{B(y,r/2)}{\bf 1}_E
			&\le C_3  \dashint_{B(y,r)}{\bf 1}_E\,d\mu\\
			&\le C_3\frac{\mu(E\cap \frac{1}{4}B_{lk})}{\mu(B(y,r))},\quad \text{since }B(y,r)\subset \frac{1}{4}B_{lk},\\
			&\le C_3C_{\mu}^{k+1}\frac{\mu(E\cap \frac{1}{4}B_{lk})}{\mu(B_{lk})},\quad \text{since }
			B(y, r)\subset B_{(l+1)k}\subset B(y, 2^{k+1}r),\\
			&\overset{\eqref{eq3.26-24Jan}}{\le}2C_1C_3C_{\mu}^{k+1}\frac{\left(2^{-(lk+k-1)}R \right)^s}{\mu(B_{lk})}\caps(W'\cap B_{lk}, 2B_{lk})\\
			&\overset{\eqref{eq4.1-28oct}}{\le }2C_1C_2C_3C_{\mu}^{k+1} \frac{\caps(W'\cap B_{lk},2B_{lk})}{\caps(B_{lk},2B_{lk})}.
		\end{align*}
		
		For each $y\in F_{(l+1)k}$, it follows that
		\[
		{\bf 1}_{E}^\vee(y)
		\le C_4\frac{\caps(W'\cap B_{lk},2B_{lk})}{{\rm cap}_{s,1}(B_{lk},2B_{lk})}\le\frac{1}{2}.
		\]
		
		By the Lebesgue's differentiation theorem, the claim holds for $l+1$ after modifying $E$ on a set of zero measure, which completes the proof.
	\end{proof}
	
	\begin{proof}[Proof of Theorem \ref{weakCartan-main}]
		Let $3< k\in\mathbb N$ be the constant from Lemma \ref{nonlocal} applied with $c_0=\frac{1}{4}$. 
		By Lemma \ref{openset}, there exists an open set $W\supset A$ that is $s$-thin at $x$.
		Then fix $R>0$ such that
		\begin{equation}\label{eq3.27-22nov}
			\sup_{0<t\leq R} \frac{\caps(W\cap B(x,t), B(x,2t))}{\caps(B(x,t),B(x,2t))}\leq \frac{1}{2C_4}
		\end{equation}
		where $C_4$ is the constant from Proposition \ref{empty-intersection}. By Lemma \ref{lem3.15-dec}, we may assume that  $P_s(B(x, 2^{-i}R))<\infty$ for each $i=0,1,\ldots$.
		Let $B_i:=B(x,2^{-i}R)$ and  $H_i:=B_i\setminus \frac{9}{10} \overline{B_{i+1}}$ where $i=0,1,\ldots$. For each $i=0,1,\ldots$, we set
		\[
		D_i:= \bigcup_{l=0}^\infty H_{i+lk}.
		\]
		Then $\bigcup_{i=0}^{k-1}D_i=B(x,R)$. Let $W_i:=W\cap D_i$ where $i=0,1,\ldots$,
		and so each $W_i$ is open.
		Moreover,
		\begin{equation}\label{eq3.28--22nov}
			\bigcup_{i=0}^{k-1}W_i=W\cap B(x,R)\supset A\cap B(x,R).
		\end{equation}
		By Lemma \ref{lemma4.6} applied to $W_i$ with $i\in \mathbb N\cup\{0\}$, there exists a $\mathcal{K}_{W_i,0}(2^{k-1}B_i)$-solution $E_i$  such that  $W_i\subset E_i\subset 2^{k-1}B_i$ and 
		\begin{align}
			\|{\bf 1}_{E_i}\|_{\dot W^{s,1}(X)}\leq\,  {\rm cap}_{s,1}(W_i,2^{k-1}B_i) 
			\overset{\eqref{e4.2-29oct}}{\leq} \ {\rm cap}_{s,1}(W_i,2B_i) \label{eq5.20-19nov}.
		\end{align}
		We will show that the sets $E_j$, with $j=0,\ldots, k-1$, satisfy our desired claims.

		Each $E_j$ is an $s$-superminimizer 
		{in $2^{k-1}B_j$, }
		for $j=0,\ldots, k-1$,
		by Proposition~\ref{prop2.4-22nov}, and hence an $s$-superminimizer 
		{in  $B$, }
		since $B \subset 2^{k-1}B_j$.
		Since each $W_j$ is open, we have ${\bf 1}_{E_j}^\wedge=1$ on $W_j$, and  hence \eqref{eq3.28--22nov} gives that 
		\[
		\max\{ {\bf 1}_{E_0}^\wedge, {\bf 1}_{E_1}^\wedge,\ldots,{\bf 1}_{E_{k-1}}^\wedge\} \equiv 1
		\]
		on $A\cap B(x,R)$.
		Hence, \ref{a-item-}-\ref{c-item-} hold.
		
		Let $F_{i}:=\frac{1}{4}B_{i}\setminus \frac{1}{2^{k-1}}B_{i}=B_{i+2}\setminus B_{i+k-1}$, where $i\in\mathbb N\cup\{0\}$.  Applying Proposition \ref{empty-intersection} with the choices
		$W'=W_i$ and $B'=B_i$,
		we may assume by possibly modifying $E_i$ on a set of zero measure that for each  $i,l\in\mathbb N\cup\{0\}$,
		\begin{align}\label{eq5.21-19nov}
			E_i\cap F_{i+lk}=\emptyset,
		\end{align}
		and
		\begin{align}
			P_s(E_i\cap B_{i+lk+k-1}) 
			&\overset{\eqref{new-2feb}}{\le} 
			\caps(W_i\cap B_{i+lk}, 2B_{i+lk})\notag\\
			&\le \caps(W\cap B_{i+lk}, 2B_{i+lk}).
			\label{eq3.32-24Jan}
		\end{align}
		
		Fix $0<\delta <1$.
		Since $W$ is $s$-thin at $x$, there exists a sufficiently large $i_\delta\in \mathbb N$ such that for all $i\in\mathbb N\cup\{0\}$ with $i\geq i_\delta$,
		\begin{align}\label{weakCartan-3}
			\frac{\caps(W\cap B_i,2B_i)}{\caps(B_i,2B_i)}\leq \delta.   
		\end{align}
		Combining this with the isoperimetric inequality \eqref{isoperimetric-set}, we have that for $j=0,\ldots,k-1$ and $l\in\mathbb N\cup\{0\}$ with $j+lk> i_{\delta}$,
		\begin{align*}
			\mu(E_j\cap 2B_{j+lk+k-1})&\overset{\eqref{eq5.21-19nov}}{=}\mu(E_j\cap B_{j+lk+k-1})\\
			&\overset{\eqref{isoperimetric-set}}{\lesssim_s} (2^{-j-lk-k+1}R)^s\|{\bf 1}_{E_j\cap B_{j+lk+k-1}}\|_{\dot W^{s,1}(X)}\\
			&\overset{\eqref{eq3.32-24Jan}}{\lesssim_k}  (2^{-j-lk}R)^s \caps(W\cap B_{j+lk},2B_{j+lk})\\
			&\overset{\eqref{weakCartan-3}}{\leq} \delta  (2^{-j-lk}R)^s \caps(B_{j+lk},2B_{j+lk})\\
			&\overset{\eqref{eq4.1-28oct}}{\simeq_s}  \delta \mu(B_{j+lk})\\
			&\simeq_k \delta\mu(2B_{j+lk+k-1}).
		\end{align*}
		Since $k$ depends only on $s$ and the doubling constant, the above estimate yields
		\begin{equation}\label{eq3.31-22dec}
			\frac{\mu(E_j\cap 2B_{j+lk+k-1})}{\mu(2B_{j+lk+k-1})}\le C_5\delta,
		\end{equation}
		where $C_5>0$ depends only on $s$ and the doubling constant.
		We now choose $\delta_0:=\left(2C_5C_{\mu}^{\lceil \log_2(50) \rceil} \right)^{-1}$.  
		Then for all $0<\delta<\delta_0$ and all $l\in \mathbb N\cup\{0\}$ with $j+lk>i_\delta$,
		\begin{align}
			\frac{\mu(E_j\cap 2B_{j+lk+k-1})}{\mu(2B_{j+lk+k-1})}
			\overset{\eqref{eq3.31-22dec}}{\le} \frac{1}{2C_\mu^{\lceil \log_2(50)\rceil}}.\label{weakCartan-4}
		\end{align}
		{By \eqref{weakCartan-4},}
		we conclude that ${\bf 1}_{E_j}^\vee(x)=0$ for each $j=0,\ldots,k-1$, which proves \ref{d-item-}.
		Recall that $I_{E_0}$ and $\partial^*E_0$ are defined  as in \eqref{eq2.19-2011}-\eqref{eq2.20-2011}.
		Notice that for $j=0,\ldots,k-1$, $l\in\mathbb N\cup\{0\}$,
		\begin{align*}
			&\quad \{{\bf 1}_{E_j}^\vee=1\}\cap B_{j+lk+k-1}\\
			&= (I_{E_j}\cup \partial^*E_j)\cap B_{j+lk+k-1}\\
			&= \left(I_{E_j\cap B_{j+lk+k-1}}\cup \partial^*(E_j\cap B_{j+lk+k-1}) \right)\cap B_{j+lk+k-1}.
		\end{align*}
		Combining this with Lemma $\ref{lemma4.11-Nov.12}$, which is applicable due to
		\eqref{weakCartan-4}, we get that for $j=0,\ldots,k-1$, $l\in\mathbb N\cup\{0\}$ with  $j+lk>i_\delta$, and $0<\delta<\delta_0$,
		\begin{align*}
			&\quad  \frac{\caps(\{{\bf 1}_{E_j}^\vee =1\}\cap B_{j+lk+k-1}, 2B_{j+lk+k-1})}{\caps(B_{j+lk+k-1},2B_{j+lk+k-1})}\\
			&\overset{\eqref{eq4.1-28oct}}{\simeq_k}\, \frac{\caps(\left(I_{E_j\cap B_{j+lk+k-1}}\cup \partial^*(E_j\cap B_{j+lk+k-1}) \right)\cap B_{j+lk+k-1}, 2B_{j+lk+k-1})}{\caps(B_{j+lk},2B_{j+lk})}\\
			&\overset{\eqref{eq2.27-22nov}}{\lesssim_k}  \, 2\frac{P_s(E_j\cap B_{j+lk+k-1})} {\caps(B_{j+lk},2B_{j+lk})}\\
			& \overset{\eqref{eq3.32-24Jan}}{\lesssim_k}   \frac{{\rm cap}_{s,1}(W\cap B_{j+lk}, 2B_{j+lk})}{\caps(B_{j+lk},2B_{j+lk})}\\
			&\overset{ \eqref{weakCartan-3}}{\leq}\delta.
		\end{align*}
		It follows that for $j=0,\ldots,k-1$,
		\[
		\lim_{l\to\infty}  \frac{\caps(\{{\bf 1}_{E_j}^\vee =1\}\cap B_{j+lk+k-1}, 2B_{j+lk+k-1})}{\caps(B_{j+lk+k-1},2B_{j+lk+k-1})}=0.
		\]
		By Lemma \ref{lem4.4-29oct}, this gives that for $j=0,\ldots,k-1$,
		\[
		\lim_{r\to 0^+} \frac{\caps(\{{\bf 1}_{E_j}^\vee =1\}
			\cap B(x,r), B(x,2r))}{\caps(B(x,r),B(x,2r))}=0,
		\]
		which implies that for $j=0,\ldots,k-1$, $\{ {\bf 1}_{E_j}^\vee>0\}$ is $s$-thin at $x$.
		By the subadditivity of ${\rm cap}_{s,1}$, and the fact that $$\left\{\max_{0\leq j\leq k-1} {\bf 1}_{E_j}^\vee >0\right\}\subset \bigcup_{0\leq j\leq k-1}\{ {\bf 1}_{E_j}^\vee>0 \},$$
		we conclude that
		\[
		\left\{\max_{0\leq j\leq k-1} {\bf 1}_{E_j}^\vee >0\right\}
		\]
		is $s$-thin at $x$, which is \ref{e-item-}.
	\end{proof}
	
	\subsection{Proof of Theorem \ref{thm3.17-26nov}}\

	Let $(X,d,\mu)$ be a metric measure space.
	A function $u:X\to[-\infty,\infty]$ is said to
	be \emph{$C_{s,1}$-quasicontinuous} if for every $\varepsilon>0$, there exists an open set
	$G\subset X$ such that $C_{s,1}(G)<\varepsilon$ and $u|_{X\setminus G}$ is continuous. 
	
	A point $x\in X$ is said to be a \emph{Lebesgue point} of a locally integrable function $f:X\to[-\infty,\infty]$
	if 
	\begin{equation}\label{eq3.34-14dec}
		\lim_{r\to 0^+}\dashint_{ B(x,r)}|f-
		{ f(x)}
		|\,d\mu=0.
	\end{equation}
	We denote by $\mathcal N_{f}$ the set of non-Lebesgue points of $f$.
	The Lebesgue representative $f^*$ of $f$ at $x\in X$ is defined by
	\[
	\limsup_{r\to0^+}\dashint_{B(x,r)}f(y)\,dy=:f^*(x).
	\]

	\begin{proposition}\label{prop:quasi.-rep.}
		
		Let  $f\in W^{s,1}(X)$.
		Then $C_{s,1}(\mathcal N_f)=0$ and 
		the Lebesgue representative  $f^*$ is $C_{s,1}$-quasicontinuous.
	\end{proposition}
	
	\begin{proof}
		By \cite[Proposition~4.1]{KYZ10}, we have $W^{s,1}(X) = M^{s}_{1,1}(X)$,
		where $M^{s}_{1,1}$ denotes the \emph{Haj{\l}asz--Triebel--Lizorkin space}.
		Moreover, the two norms  $\|f\|_{W^{s,1}(X)}$
		and $\|f\|_{M^{s}_{1,1}(X)}$ are comparable, with comparisons constants depending only
		on $s$ and the doubling constant.
		In particular, the capacity associated with the $M^{s}_{1,1}$-norm is comparable with
		$C_{s,1}$.
		By \cite[Theorem~8.1]{HKT17}, the claim holds for $M^{s}_{1,1}$-functions, and hence it also holds for $W^{s,1}$-functions.
	\end{proof}
	See also \cite[Corollary~1.2]{BBS22} for an alternative proof of the above proposition
	on compact doubling metric spaces.
	
	\begin{lemma}
		\label{lem3.16-26nov}
		Let 
		$0<R_0<\infty$.
		Let $A\subset X$ and $x\in X$ such that $A$ is $s$-thin at $x$.
		Then
		\[
		\lim_{r\to0^+}{\rm cap}_{s,1}(A\cap B(x,r), B(x, R_0))=0.
		\] 
	\end{lemma}
	\begin{proof}
		Suppose that $C_{s,1}(\{x\})=0$. Then 
		by \eqref{eq4.7-29oct} and by the outer capacity property, the claim is obtained since
		\begin{align*}
			\limsup_{r\to0^+} {\rm cap}_{s,1}(A\cap B(x,r),B(x,R_0))
			&\lesssim_s \left(1+\frac{1}{R_0^s}\right) \limsup_{r\to0^+} C_{s,1}(A\cap B(x,r))\\
			&\leq \left(1+\frac{1}{R_0^s}\right) \limsup_{r\to0^+} C_{s,1}(B(x,r))\\
			&=\left(1+\frac{1}{R_0^s}\right)\limsup_{r\to0^+} C_{s,1}(\{ x\})=0.
		\end{align*}
		We next consider the case $C_{s,1}(\{x\})>0$.
		By Lemma \ref{lem2.12-15dec}, we know that
		$x\not\in A$.
		Let $0<\varepsilon< C_{s,1}(\{x\})$ be arbitrary.
		By Theorem \ref{weakCartan-main}, there exist $R>0$ and finitely many functions
		$u_0,u_2,\ldots,u_{k-1} \in W^{s,1}(X)$, for some $k\in\mathbb N$, such that
		$\max_{0\le j\le k-1} u_j^\wedge \equiv 1$ on $A\cap B(x,R)$,
		and $u_j^\vee(x)=0$ for all $j=0,\ldots,k-1$.
		Since the functions take values between $0$ and $1$, then also
		$\max_{0\le j\le k-1} u_j^\vee \equiv 1$ on $A\cap B(x,R)$.
		
		By Proposition \ref{prop:quasi.-rep.}, there is a set $G$ such that
		$C_{s,1}(G)<\varepsilon$ and each $u^*_j|_{X\setminus G}$ is continuous;
		moreover, we can assume that all points in $X\setminus G$ are Lebesgue points
		of $u_j$ for all $j$, so that $u_j^*=u_j^{\vee}$ in $X\setminus G$.
		In particular, each $u^*_j|_{X\setminus G}$ is continuous at $x$ 
		{ and $u^*_j(x)=u_j^\vee(x)=0$,
		}
		since $0<\varepsilon< C_{s,1}(\{x\})$.
		Thus for a sufficiently small $r_0\in (0,R_0)$,  for all $0<r<r_0$ we have 
		\[
		\{ 
		{ \cup_{j=0}^{k-1}\{ u_j^\vee = 1\}}
		\} \cap B(x,r) \subset G,
		\]
		and so 
		\[
		C_{s,1}(
		{\cup_{j=0}^{k-1}\{ u_j^\vee = 1\}\cap B(x,r)}
		) \leq C_{s,1}(G) <\varepsilon.
		\]
		Since 
		{$\max_{0\leq j\leq k-1} u_j^\vee\equiv 1$}
		on $A\cap B(x,R)$,  the above gives that $C_{s,1}(A\cap B(x,r))<\varepsilon$ for all $0<r<r_0$.
		We obtain from \eqref{eq4.7-29oct} that
		\[
		\ {\rm cap}_{s,1}(A\cap B(x,r),B(x,R_0))
		\lesssim_s \left(1+\frac{1}{R_0^s}\right)\varepsilon.
		\] 
		Letting $\varepsilon\to 0$, we obtain the  claim. 
	\end{proof}
	
	\begin{proof}[Proof of Theorem \ref{thm3.17-26nov}]
		Let $0<R_0<\frac{1}{4}\diam X$ be arbitrary.
		By Lemma \ref{lem3.16-26nov}, we have 
		$$\lim_{r\to0^+}\caps(A\cap B(x,r), B(x, R_0))=0.$$
		Then there exist a decreasing sequence $r_j$ with $r_j\to0^+$ and a sequence of open sets $U_j$  with $A\cap B(x,r_j)\subset U_j$
		such that 
		\begin{equation}\notag
			{\rm cap}_{s,1}(U_j, B(x,R_0))< \frac{\varepsilon}{j\cdot 2^j}.
		\end{equation}
		By Lemma \ref{lemma4.6}, for each $j$ there is 
		a $\mathcal K_{U_j,0}(B(x,R_0))$-solution  $f_j$ such that $0\leq f_j\leq 1$,
		$f_j=0$ on $X\setminus B(x,R_0)$ and  $f_j=1$ in $U_j$,
		and $\|f_j\|_{\dot W^{s,1}(X)}\leq {\rm cap}_{s,1}(U_j, B(x,R_0))$.
		By the Sobolev-type inequality \eqref{isoperimetric-func.}, it follows that 
		\begin{align*}
			\|f_j\|_{W^{s,1}(X)}
			&=\|f_j\|_{L^1(X)}+\|f_j\|_{\dot W^{s,1}(X)}\\
			&\lesssim_{R_0} \|f_j\|_{\dot W^{s,1}(X)}\\
			&\leq  \caps(U_j, B(x,R_0))\\
			&\leq \frac{\varepsilon}{j\cdot 2^j}.  
		\end{align*}
		{For $y\in X$, we set
			\[
			f(y):= \begin{cases}
				\sum_{j\in\mathbb N} j\cdot f_j(y) &\text{\ if $y\neq x$},\\
				0	&\text{\ if $y= x$}.
			\end{cases}
			\]
			Since $W^{s,1}(X)$ is a Banach space by \cite[Remark 9.8]{BBS22}, we obtain that  $f\in W^{s,1}(X)$.
		}
		It follows that $f=0$ on $X\setminus B(x,R_0)$, {  $0\leq f<\infty$ on $X$, and  $f(y)\geq j$ for $y\in A\cap B(x,r_j)\setminus \{x\}$.}
		Moreover, $f\in \mathcal K_{f,0}(B(x,R_0))$ because $f=0$ on $X\setminus B(x,R_0)$, $f\in W^{s,1}(X)$, {and $0\leq f<\infty$ on $X$}. Thus $\mathcal K_{f,0}(B(x,R_0))\neq\emptyset$, and we obtain from
		Lemma \ref{prop2.4-22nov}
		that there is an $s$-superminimizer $u\in\mathcal K_{f,0}(B(x,R_0))$ on $B(x,R_0)$  such that $u\in W^{s,1}(X)$, $u\geq f$ a.e. on $B(x,R_0)$ and $u=0$ a.e. on $X\setminus B(x,R_0)$, and so
		$
		\lim_{A\ni y\to x}u^\wedge (y)=\infty.
		$
		
		It suffices to prove that $u^\vee(x)<\infty$.
		Since $C_{s,1}(\{x\})>0$, Proposition~\ref{prop:quasi.-rep.} yields that $x$ is a Lebesgue point of $u$, so that $u^\wedge(x)=u^\vee(x)=u^*(x)\in\mathbb R$, and the desired result follows.
	\end{proof}

	
	\
	
	\
	
	\

	{\bf Conflict of interest statement}: None declared.
	
	{\bf Data availability}: 
	No data was used for the research described in the article.
	
	
	\newcommand{\etalchar}[1]{$^{#1}$}


\begin{thebibliography}{99}
		
		
		
		\bibitem[BB11]{BB11}A. Bj\"orn and J. Bj\"orn,
		\textit{ Nonlinear potential theory on metric spaces},
		European Mathematical Society (EMS), Zürich, 2011, xii+403 pp.
		
		\bibitem[BB23]{BB23}A. Bj\"orn and J. Bj\"orn,
		\textit{Sharp Besov capacity estimates for annuli in metric spaces with doubling measures}, 
		Math. Z. 305 (2023), no. 3, Paper No. 41, 26 pp.
		
		\bibitem[BBL18]{BBL18}A. Bj\"orn, J. Bj\"orn, and V. Latvala,
		\textit{The Cartan, Choquet and Kellogg properties for the fine topology on metric spaces},
		J. Anal. Math, 135 (2018), 59--83. 
		
		\bibitem[BBL15]{BBL15}A. Bj\"orn, J. Bj\"orn, and V. Latvala,
		\textit{The weak Cartan property for the $p$-fine topology on metric spaces},
		Indiana University Mathematics Journal 64, no. 3 (2015), 915--941.
		
		\bibitem[BBS22]{BBS22}A. Bj\"orn, J. Bj\"orn, and N. Shanmugalingam,
		\textit{Extension and trace results for doubling metric measure spaces and their hyperbolic fillings},
		J. Math. Pures Appl. 159 (2022) 196--249.
		
		\bibitem[CRS10]{CRS10}L. Caffarelli, J.-M.  Roquejoffre, and O. Savin,
		\textit{Nonlocal minimal surfaces},
		Comm. Pure Appl. Math. 63 (2010), no. 9, 1111--1144.
		
		\bibitem[DKP16]{DKP16}A. Di Castro, T. Kuusi, and G. Palatucci,
		\textit{Local behavior of fractional  $p$-minimizers},
		Ann. Inst. H. Poincaré C Anal. Non Linéaire 33 (2016), no. 5, 1279--1299.
		
		\bibitem[DLV23]{DLV23}B. Dyda, J. Lehrb\"ack, and A. V\"ah\"akangas,
		\textit{Fractional Poincar\'e and localized Hardy inequalities on metric spaces},
		Adv. Calc. Var. 16, (2023) no. 4, 867--884.
		
		\bibitem[Ha03]{Ha03}P. Hajłasz,
		\textit{Sobolev spaces on metric-measure spaces},
		Heat kernels and analysis on manifolds, graphs, and metric spaces (Paris, 2002), 173--218.
		Contemp. Math., 338,
		American Mathematical Society, Providence, RI, 2003.
		
		\bibitem[HK00]{HK00}P. Hajłasz and P.  Koskela,
		\textit{Sobolev met Poincar\'e},
		Mem. Amer. Math. Soc. 145 (2000), no. 688, x+101 pp.
		
		\bibitem[HKT17]{HKT17}T. Heikkinen, P. Koskela, and H. Tuominen,
		\textit{Approximation and quasicontinuity of Besov and Triebel-Lizorkin functions},
		Transactions of the American Mathematical Society (2017), 369(5), 3547--3573.
		
		\bibitem[HKST15]{HKST15}J. Heinonen, P. Koskela, N. Shanmugalingam, and J. Tyson,
		\textit{Sobolev spaces on metric measure spaces. An approach based on upper gradients},
		New Mathematical Monographs, 27. Cambridge University Press, Cambridge, 2015. xii+434 pp.
		
		\bibitem[HKM06]{HKM06}J. Heinonen, T. Kilpeläinen, and O. Martio,
		\textit{Nonlinear Potential Theory of Degenerate Elliptic Equations},
		2nd ed., Dover, Mineola, NY, 2006. 
		
		\bibitem[Kl25]{Kl25}J. Kline,
		\textit{Regularity of sets of finite fractional perimeter and nonlocal minimal surfaces in metric measure spaces},
		J. Geom. Anal. 35 (2025), no. 6, Paper No. 182, 44 pp.
		
		\bibitem[KLLZ25]{KLLZ25}J. Kline, P. Lahti, J. Li, and X. Zhou,
		\textit{Geometric inequalities related to fractional perimeter: fractional Poincaré, isoperimetric, and boxing inequalities in metric measure spaces},
		arXiv:2511.04187.
		
		\bibitem[KYZ11]{KYZ10}P. Koskela, D. Yang, and Y. Zhou,
		\textit{Pointwise characterizations of Besov and Triebel–Lizorkin spaces and quasiconformal mappings}, Advances in Mathematics, Volume 226, Issue 4, 2011, Pages 3579--3621.
		
		\bibitem[La20]{La20}P. Lahti,
		\textit{Superminimizers and a weak Cartan property for p = 1 in metric spaces},
		J. Anal. Math. 140 (2020), no. 1, 55--87.	
		
		
		
		\bibitem[MZ97]{MZ}J. Mal\'{y} and W. Ziemer,
		\textit{Fine regularity of solutions of elliptic partial differential equations},
		Mathematical Surveys and Monographs, 51. American Mathematical Society, Providence, RI, 1997. xiv+291 pp.
		
		
		\bibitem[Sh00]{Sh00}N. Shanmugalingam,
		\textit{Newtonian spaces: an extension of Sobolev spaces to metric measure spaces},
		Rev. Mat. Iberoamericana 16 (2000), no. 2, 243--279.
		
		\bibitem[Vi91]{Vi91}A. Visintin,
		\textit{Generalized coarea formula and fractal sets},
		Japan J. Indust. Appl. Math. 8 (1991), 175--201
		
	\end{thebibliography}
\end{document}